\documentclass[a4paper,11pt,reqno]{amsart}

\usepackage{url}
\usepackage[colorlinks=true, linkcolor=blue, citecolor=ForestGreen, urlcolor=blue]{hyperref}
\usepackage{cite}

\usepackage{latexsym,amsmath,amssymb,amsfonts,amsthm}
\usepackage{mathrsfs}
\usepackage{mathtools}
\usepackage{dsfont}
\usepackage{bbm}
\usepackage{soul}
\usepackage{cancel}
\usepackage[utf8]{inputenc}

\usepackage[usenames, dvipsnames]{color}
\usepackage[usenames, dvipsnames, cmyk]{xcolor}
\usepackage{graphicx}
\usepackage{multirow}
\usepackage{enumerate}
\usepackage{enumitem}
\usepackage{subcaption}

\usepackage{a4wide}
\usepackage{epsfig,epstopdf}

\allowdisplaybreaks
\numberwithin{equation}{section}

\definecolor{grey}{rgb}{.7,.7,.7}
\definecolor{refkey}{gray}{.45}
\definecolor{labelkey}{gray}{.45}

\newcommand{\xupref}[2]{\hspace{-0.3ex}\stackrel{\eqref{#1}}{#2}} 

\newtheorem{theorem}{Theorem}[section]
\newtheorem{proposition}[theorem]{Proposition}
\newtheorem{lemma}[theorem]{Lemma}

\theoremstyle{remark}
\newtheorem{remark}[theorem]{Remark}
\theoremstyle{definition}
\newtheorem{definition}[theorem]{Definition}
\newtheorem{example}[theorem]{Example}

\usepackage{tikz}
\usepackage{pgf,pgfplots}
\usetikzlibrary{calc}
\usetikzlibrary{decorations.markings}
\usetikzlibrary{decorations.pathmorphing}
\usetikzlibrary{decorations.shapes}
\usetikzlibrary{shapes,arrows,shapes.geometric,patterns,fadings}
\usetikzlibrary{pgfplots.fillbetween}

\newcommand{\e}{\varepsilon}

\newcommand{\N}{\mathbb N}

\newcommand{\R}{\mathbb R}

\newcommand{\locto}{\stackrel{\mathrm{loc}}{\to}}
\newcommand{\dd}{\,\mathrm{d}}

\newcommand{\Hone}{\mathcal{H}^1}
\newcommand{\Hd}{\mathcal{H}^{d-1}}
\renewcommand{\setminus}{\backslash}

\newcommand{\defeq}{\coloneqq}

\newcommand{\per}{\mathcal{P}}

\newcommand{\defi}[2]{\delta(#1,#2)}
\newcommand{\asymm}[2]{#1\triangle#2}
\newcommand{\Asymm}[2]{\Delta(#1,#2)}
\newcommand{\dis}[2]{\mathrm{d}(#1,#2)}
\newcommand{\comp}{\mathfrak{M}}
\newcommand{\Rba}{\overline{R}}

\newcommand{\E}{\mathcal{E}}
\newcommand{\F}{\mathcal{F}}
\newcommand{\G}{\mathcal{G}}
\newcommand{\lens}{L_{m}}
\newcommand{\lensc}{\mathcal{L}_{m}}
\newcommand{\fr}{{\textstyle\frac{\sqrt{3}}{2}}}
\newcommand{\Fun}{\mathscr{F}}

\newcommand{\ba}{\begin{array}}
\newcommand{\ea}{\end{array}}

\newcommand{\bthm}{\begin{theorem}}
\newcommand{\ethm}{\end{theorem}}
\newcommand{\bprop}{\begin{proposition}}
\newcommand{\eprop}{\end{proposition}}
\newcommand{\blemma}{\begin{lemma}}
\newcommand{\elemma}{\end{lemma}}
\newcommand{\bexmpl}{\begin{example}}
\newcommand{\eexmpl}{\end{example}}

\newcommand{\beqn}{\begin{equation}}
\newcommand{\eeqn}{\end{equation}}
\newcommand{\beqns}{\begin{equation*}}
\newcommand{\eeqns}{\end{equation*}}

\newcommand{\supp}{\operatorname{supp}}

\newcommand{\pr}{\prime}
\newcommand{\pt}{\partial}

\renewcommand{\leq}{\leqslant}
\renewcommand{\geq}{\geqslant}

\definecolor{mygreen}{rgb}{0.1,0.75,0.2}

\newcommand{\ssubset}{\subset\!\subset}

\newcounter{myenumi}
\DeclareMathOperator{\id}{Id}

\DeclareMathOperator{\dive}{div}

\DeclareMathOperator*{\dist}{dist}

\newcommand{\etalchar}[1]{$^{#1}$}

\title[A stability inequality for the planar lens partition]{A stability inequality for the planar lens partition}

\author{Marco Bonacini}
\address[Marco Bonacini]{Department of Mathematics, University of Trento, Italy}
\email{marco.bonacini@unitn.it}

\author{Riccardo Cristoferi}
\address[Riccardo Cristoferi]{Department of Mathematics - IMAPP, Radboud University, Nijmegen, The Netherlands}
\email{riccardo.cristoferi@ru.nl}

\author{Ihsan Topaloglu}
\address[Ihsan Topaloglu]{Department of Mathematics and Applied Mathematics, Virginia Commonwealth University, Richmond, VA, USA}
\email{iatopaloglu@vcu.edu}

\date{\today}       

\thanks{This is a post-peer-review, pre-copyedit version of an article published in Proceedings of the Royal Society of Edinburgh: Section A Mathematics. The final authenticated version is available online at: \url{https://doi.org/10.1017/prm.2025.2}.}                                 


\begin{document}

\begin{abstract}
Recently it has been shown that the unique locally perimeter minimizing partitioning of the plane into three regions, where one region has finite area and the other two have infinite measure, is given by the so-called standard lens partition. Here we prove a sharp stability inequality for the standard lens; hence strengthening the local minimality of the lens partition in a quantitative form. As an application of this stability result we consider a nonlocal perturbation of an isoperimetric problem.
\end{abstract}

\maketitle



\section{Introduction}\label{sec:intro}

The goal of the classical cluster problem in $\R^d$ is to find a configuration of $N$ regions with fixed finite $d$-dimensional volumes and an exterior region of infinite volume so that the total surface measure of the interfaces between the regions is minimal (see, for example, \cite[Part~IV]{Maggi}). For $N \leq d + 1$ and for any given collection of $N$ finite volumes, the \textit{standard $N$-bubble} is conjectured to be the unique minimizer (up to isometries) of the cluster problem. If $N=1$, then the problem reduces to the isoperimetric problem. For any given pair of volumes, the double bubble problem (i.e.\ when $N$=2) has been studied extensively: Foisy, Alfaro, Brock, Hodges, and Zimba proved the minimality of the standard double bubble in $\R^2$ \cite{FoiAlfBroHodZim93}; this was extended to $\R^3$ by Hutchings, Morgan, Ritor\'{e}, and Ros \cite{HutMorRitRos02}; to $\R^4$ by Reichardt, Heilmann, Lai, and Spielman \cite{ReiHeiLaiSpi03}, and to any $\R^d$ by Reichardt \cite{Rei08}. For $N=3$ Wichiramala \cite{Wic04} proves that triple bubbles are the isoperimetric clusters in $\R^2$. Recently, Milman and Neeman confirmed the double bubble conjectures for $d \geq 2$, the triple bubble conjectures for $d \geq 3$ and the quadruple bubble conjectures for $d \geq 4$ in \cite{MilNee}.

A variant of the classical cluster problem is to characterize locally isoperimetric $N$-partitions with more than one region having infinite volume. Such partitions divide the space into $N$ regions with prescribed (finite or infinite) volume and locally minimize the surface measure of the interfaces with respect to all compactly supported variations that also preserve the volume of each region. When two (or more) regions have infinite volume the measure of their interface is also infinite. Therefore, for such partitions, one needs to consider locally minimizing configurations in the following sense: For every ball $B_R$ of radius $R$ the perimeter of the partition in the interior of $B_R$ is minimal among all partitions with regions of the same prescribed volumes as the original partition, but whose difference, in the set-theoretic sense, with the corresponding regions of the original partition is compactly contained in $B_R$.

The study of locally isoperimetric partitions with more than one infinite region has only recently been initiated by Alama, Bronsard, and Vriend \cite{AlaBroVri}. They characterize the unique locally isoperimetric partition of the plane into three regions with one region of given fixed area, and the other two having infinite area, as the \textit{standard lens partition}. Novaga, Paolini, and Tortorelli \cite{NovPaoTor} further this study by obtaining a general closure theorem for limits of sequences of locally isoperimetric partitions, showing that they are themselves local minimizers, provided that they have flat interfaces outside some compact set. This enables them to identify several locally isoperimetric partitions in $\R^d$. In two dimensions they prove that any planar locally isoperimetric partition has at most 3 chambers with infinite area. They also give a complete characterization of planar local minimizers in the case the total number of finite and infinite regions does not exceed 4 as either the \textit{lens partition} (1 finite, 2 infinite regions), the \textit{peanut} (2 finite, 2 infinite regions), or the \textit{Reuleaux triangle} (1 finite, 3 infinite regions), where the last two were conjectured to be local minimizers in \cite{AlaBroVri} (see Figure~\ref{fig:intro}). Finally, in \cite{BroNov}, Bronsard and Novack study a partitioning of $\R^d$ into 1 finite and 2 infinite regions where the surface measure between different pairs of regions is computed with respect to some given weights. After establishing that the \textit{standard weighted lens cluster} is locally minimizing under the standard positivity and triangularity conditions on the weights, they also prove its uniqueness under some additional symmetry and growth assumptions on the weights.

\begin{figure}[htbp]
\centering
\begin{subfigure}[c]{.4\linewidth}
    \centering
\begin{tikzpicture}[scale=1.6,line cap=round,line join=round]
    \fill[green!30!white] (0,{sin(30)}) -- ({0.15*cos(60)},{sin(30)+0.15*sin(60)}) arc (60:180:0.15);    
    \begin{scope}
        \fill[blue!15!white] ({cos(30)},0) circle (1);
    \end{scope}
    \begin{scope}
        \clip (0,0) circle (0.5);
        \fill[orange!50!white] ({-cos(30)},0) circle (1);
    \end{scope}    
    \draw[very thick] (0,{-sin(30)}) arc (-30:30:1);
    \draw[very thick] (0,{sin(30)}) arc (90:270:0.5);
    \draw[very thick] (0,{-sin(30)}) arc (-150:150:1);
    \draw [dashed] (-0.5,0) -- ({1+cos(30)},0);
    \draw ({0.15*cos(60)},{sin(30)+0.15*sin(60)}) arc (60:180:0.15);
	\draw [dashed] (0,{sin(30)}) -- (-0.5,{sin(30)});  
	\draw [dashed] (0,{sin(30)}) -- ({0.5*cos(60)},{sin(30)+0.5*sin(60)}); 
	\draw [dashed] (0,{sin(30)}) -- ({0.5*cos(300)},{sin(30)+0.5*sin(300)});
\end{tikzpicture}    
\end{subfigure}\hspace{0.3cm}
\begin{subfigure}[c]{.4\linewidth}
    \centering
\begin{tikzpicture}[scale=2,line cap=round,line join=round]
	\fill[blue!15!white] (-{sin(120)},0) arc (270-120:120-90:1) arc (450-120:90+120:1);
    \fill[green!30!white] ({sin(120)},0) -- ({sin(120)+0.1},0) arc (0:120:0.1) -- ({cos(120)/2+sin(120)},{sin(120)/2});    
	\draw[very thick] (-{sin(120)},0) arc (270-120:120-90:1);
	\draw[very thick] ({sin(120)},0) arc (450-120:90+120:1);	
	\draw[very thick] (-1.2,0) -- (-{sin(120)},0);
	\draw[very thick] ({sin(120)},0) -- (1.2,0);
	\draw[very thick,dashed] (-1.5,0) -- (-1.2,0); 
    \draw[very thick,dashed] (1.2,0) -- (1.5,0);   
	\draw[dashed] (-{sin(120)},0) -- ({sin(120)},0);
    \draw ({sin(120)+0.1},0) arc (0:120:0.1);
	\draw [dashed] ({sin(120)},0) -- ({cos(120)/2+sin(120)},{sin(120)/2}); 
    \draw [dashed] ({sin(120)},0) -- ({cos(120)/2+sin(120)},{-sin(120)/2}); 
\end{tikzpicture}    
\end{subfigure}

\hspace{.7cm}
\begin{subfigure}[b]{.4\linewidth}
    \centering
\begin{tikzpicture}[scale=2.5,line cap=round,line join=round]
    \begin{scope}
        \clip ({-1+sqrt(3)/(2*sqrt(2)*cos(15)},{-1/(2*sqrt(2)*cos(15)}) circle ({1/(sqrt(2)*cos(15)});
        \clip ({-1+sqrt(3)/(2*sqrt(2)*cos(15)},{1/(2*sqrt(2)*cos(15)}) circle ({1/(sqrt(2)*cos(15)});        
        \fill[blue!15!white] (-1,-1) rectangle (0,1);
    \end{scope}
    \begin{scope}
        \clip ({1-sqrt(3)/(2*sqrt(2)*cos(15)},{-1/(2*sqrt(2)*cos(15)}) circle ({1/(sqrt(2)*cos(15)});
        \clip ({1-sqrt(3)/(2*sqrt(2)*cos(15)},{1/(2*sqrt(2)*cos(15)}) circle ({1/(sqrt(2)*cos(15)});        
        \fill[orange!50!white] (0,-1) rectangle (1,1);
    \end{scope}    
    \fill[green!30!white] (1,0) -- (1.1,0) arc (0:120:0.1);
    \fill[green!30!white] (0,{tan(15)}) -- ({0.1*cos(30)},{tan(15)+0.1*sin(30)}) arc (30:150:0.1);
    \draw[very thick] (-1,0) arc (150:60:{1/(sqrt(2)*cos(15)});
    \draw[very thick] (-1,0) arc (210:300:{1/(sqrt(2)*cos(15)});    
    \draw[very thick] (1,0) arc (30:120:{1/(sqrt(2)*cos(15)});        
    \draw[very thick] (1,0) arc (330:240:{1/(sqrt(2)*cos(15)});        
    \draw[very thick] (0,{-tan(15)}) -- (0,{tan(15)});
	\draw[very thick] (-1.2,0) -- (-1,0);
	\draw[very thick] (1,0) -- (1.2,0);
	\draw[very thick, dashed] (-1.5,0) -- (-1.2,0);
    \draw[very thick, dashed] (1.2,0) -- (1.5,0);
    \draw[dashed] (-1,0) -- (1,0);
	\draw [dashed] (1,0) -- ({1+0.4*cos(120)},{0.4*sin(120)}); 
	\draw [dashed] (1,0) -- ({1+0.4*cos(120)},{-0.4*sin(120)}); 
    \draw (1.1,0) arc (0:120:0.1); 
	\draw [dashed] (0,{tan(15)}) -- ({0.4*cos(30)},{tan(15)+0.4*sin(30)}); 
    \draw [dashed] (0,{tan(15)}) -- ({0.4*cos(150)},{tan(15)+0.4*sin(150)}); 
    \draw ({0.1*cos(30)},{tan(15)+0.1*sin(30)}) arc (30:150:0.1);     
\end{tikzpicture}    
\end{subfigure}\hspace{2cm}
\begin{subfigure}[b]{.4\linewidth}
    \centering
\begin{tikzpicture}[scale=1.6,line cap=round,line join=round]
    \begin{scope}
        \clip (1,0) circle ({sqrt(3)});
        \clip ({cos(120)},{sin(120)}) circle ({sqrt(3)});
        \clip ({cos(240)},{sin(240)}) circle ({sqrt(3)});
        \fill[blue!15!white]  (0,0) circle (3);
    \end{scope}
    \fill[green!30!white] (1,0) -- (1.15,0) arc (0:120:0.15) -- ({1+0.7*cos(120)},{0.7*sin(120)});    
    \draw[dashed] (0,0) -- (1,0);
    \draw[very thick] (1,0) -- (1.4,0);
    \draw[very thick,dashed] (1.4,0) -- (1.8,0);
    \draw[dashed] (0,0) -- ({cos(120)},{sin(120)});
    \draw[very thick] ({cos(120)},{sin(120)}) -- ({1.4*cos(120)},{1.4*sin(120)});    
    \draw[very thick,dashed] ({1.4*cos(120)},{1.4*sin(120)}) -- ({1.8*cos(120)},{1.8*sin(120)});        
    \draw[dashed] (0,0) -- ({cos(240)},{sin(240)});             
    \draw[very thick] ({cos(240)},{sin(240)}) -- ({1.4*cos(240)},{1.4*sin(240)});        
    \draw[very thick,dashed] ({1.4*cos(240)},{1.4*sin(240)}) -- ({1.8*cos(240)},{1.8*sin(240)});
	\draw[very thick] (1,0) arc (30:90:{sqrt(3)});    
	\draw[very thick] ({cos(120)},{sin(120)}) arc (150:210:{sqrt(3)});    
	\draw[very thick] ({cos(240)},{sin(240)}) arc (270:330:{sqrt(3)});     
	\draw [dashed] (1,0) -- ({1+0.7*cos(120)},{0.7*sin(120)}); 
	\draw [dashed] (1,0) -- ({1+0.7*cos(120)},{-0.7*sin(120)}); 
    \draw (1.15,0) arc (0:120:0.15); 
\end{tikzpicture}
\end{subfigure}
\caption{Some planar locally isoperimetric partitions: the standard double bubble, the standard lens, the peanut, and the Reuleaux triangle. All the highlighted angles are 120 degree angles.}
\label{fig:intro}
\end{figure}

In this paper we are interested in the stability of locally isoperimetric partitions in the spirit of the quantitative isoperimetric inequality, where a suitable distance of a set from a ball of the same volume is controlled in terms of the difference of the perimeter of the set and the perimeter of the ball (see \cite{Fus15,Mag08} for two excellent reviews). For clusters, the only stability result that we are aware of is by Cicalese, Leonardi, and Maggi \cite{CicLeoMag17} where they obtain the stability of the planar standard double bubble. Closely related is the proof by Caroccia and Maggi in \cite{CarMag16} of a quantitative version of the minimality of the honeycomb tiling of the plane. The approach adopted in both results is based on the \emph{selection principle} devised by Cicalese and Leonardi \cite{CicLeo12}, and utilizes, as an essential tool, an \textit{improved convergence} theorem for bubble clusters by Cicalese, Leonardi and Maggi \cite{CicLeoMag16,LeoMag17}. Similar ideas were also developed independently by Fusco and Morini in \cite{FusMor12}.

We follow a similar strategy and prove in our main result (Theorem~\ref{thm:stability-lens}) the stability of the planar standard lens cluster when the interfaces between regions are weighted equally. The core idea behind the proof strategy is a contradiction argument, where one assumes the existence of a sequence of partitions which violate the stability inequality, and converge locally in $L^1$ to the lens partition. Then one proceeds in two steps. First, the selection principle allows one to replace the previous sequence by a sequence of \textit{quasi-minimizing} partitions (see Definition~\ref{def:quasi-min}), which still violate the stability inequality, and also have better regularity properties. Then the improved convergence theorem yields that these new partitions converge in a stronger sense, and in particular that they are small $C^1$-perturbations of the lens partition. Therefore this reduces the proof of the stability inequality to a class of smooth perturbations of the lens. This is precisely the content of the second step of the proof, which is obtained by a ``Fuglede-type argument''.

We highlight that only in the second step of this strategy we make essential use of the specific structure of the lens partition, whereas the first step is carried out for any planar locally isoperimetric partition (in the case of a single region with infinite volume, this was already established in \cite{CicLeoMag17}). This paves the way for the proof of the stability of other locally isoperimetric partitions such as the peanut and the Reuleaux triangle, which will be the object of future investigation.

Finally, we would like to mention that the standard lens cluster is directly related to the classical problem of finding the equilibrium shapes of liquid drops confined in a half-space in the absence of gravity (see \cite[Chapter 19]{Maggi} as well as \cite{AlaBroVri2}). The stability of minimizing shapes for liquid drops has only recently been established by Pascale and Pozzetta \cite{PasPoz}. Furthermore, in \cite{Pas}, Pascale studies classical capillarity problems with the inclusion of nonlocal repulsion and gravity terms, and obtains the existence and nonexistence of minimizers. In the final section of this paper, we also study a nonlocal perturbation of an isoperimetric problem where perimeter term is related to both capillarity problems and to the partitioning problem studied here. Exploiting the stability result for the standard lens cluster, we show that the minimizers of this nonlocal problem are close (in the $L^1$-sense) to the standard lens cluster in certain parameter regimes.

\medskip
The paper is organized as follows. Section~\ref{sec:main} contains the general formulation of the problem of locally isoperimetric partitions, and the statement of our main result on the stability of the lens partition, Theorem~\ref{thm:stability-lens}. The proof is carried out in two steps as described above in Section~\ref{sec:step2} and Section~\ref{sec:step1}, respectively. Eventually, in Section~\ref{sect:nonlocal-lens} we discuss an application of the stability to a partitioning problem perturbed by a nonlocal interaction.


\section{Definitions and main result}\label{sec:main}


\subsection{Locally isoperimetric partitions: definitions} \label{subsec:def}
We start by fixing the notation and by formulating in any dimension the notion of \emph{locally isoperimetric partitions} introduced in \cite{AlaBroVri}; we follow in particular the presentation in \cite{NovPaoTor}.

Given a measurable set $E\subset\R^d$, $d\geq2$, we denote by $|E|$ its $d$-dimensional Lebesgue measure. The open ball of $\R^d$ of radius $r>0$ centered at $x_0\in\R^d$ is denoted by $B_r(x_0)$, and we simply write $B_r$ when the center is at the origin. If $E$ is a set of locally finite perimeter (in the sense of Caccioppoli--De Giorgi), we denote by $\partial^*E$ its reduced boundary, by $\nu_E$ the measure-theoretic outer unit normal, and by $\per(E;\Omega)=\Hd(\partial^*E\cap\Omega)$ its relative perimeter in a Borel set $\Omega\subset\R^d$, with $\per(E)\defeq\per(E;\R^d)$. For a set of finite perimeter $E\subset\R^d$ we adopt the convention
\begin{equation} \label{eq:topboundary}
\partial E = \bigl\{ x\in\R^d \;:\; 0<|E\cap B_r(x)| < |B_r| \text{ for all }r>0 \bigr\},
\end{equation}
which can be always assumed up to modifying $E$ in a Lebesgue-negligible set, see for instance \cite[Proposition~12.19]{Maggi}. The symmetric difference of two sets $E,F\subset\R^d$ is denoted by $\asymm{E}{F}\defeq (E\setminus F)\cup(F\setminus E)$. Given an open set $\Omega\subset\R^d$, we say that a sequence of measurable sets $(E_n)_{n\in\N}$ converge to a set $E$ in $\Omega$ if $|(\asymm{E_n}{E})\cap\Omega|\to0$, that is, if the characteristic functions $\chi_{E_n}$ converge to $\chi_E$ in $L^1(\Omega)$. We say that $E_n\to E$ locally in $\R^d$ if $E_n\to E$ in $B_R$, for every $R>0$.

\begin{definition}[Partition] \label{def:partition}
A \emph{$N$-partition} of $\R^d$, $N\geq2$, is an $N$-tuple $\E=(\E(1),\ldots,\E(N))$ of sets of locally finite perimeter in $\R^d$ such that $0<|\E(i)|\leq\infty$, $|\E(i)\cap\E(j)|=0$ for all $i,j\in\{1,\ldots,N\}$, $i\neq j$, and $|\R^d\setminus\bigcup_{i=1}^N\E(i)|=0$. 
\end{definition}

We denote the interfaces between the different regions of a partition by
\begin{equation} \label{eq:interfaces}
\E(i,j)\defeq \partial^*\E(i)\cap\partial^*\E(j), \qquad i,j\in\{1,\ldots,N\},
\end{equation}
and the boundary, the reduced boundary, and the singular set of the partition respectively by
\begin{equation} \label{eq:boundary}
\partial\E\defeq \bigcup_{i=1}^N \partial\E(i), \qquad \partial^*\E\defeq \bigcup_{i=1}^N \partial^*\E(i), \qquad \Sigma(\E)\defeq \partial\E\setminus\partial^*\E.
\end{equation}
The \emph{perimeter} of a partition $\E$ relative to a Borel set $\Omega\subset\R^d$ is defined as
\begin{equation} \label{eq:perimeter}
\per(\E;\Omega) \defeq \frac12\sum_{i=1}^N \per(\E(i);\Omega) = \sum_{1\leq i < j \leq N} \Hd(\E(i,j)\cap\Omega).
\end{equation}
We say that a sequence of partitions $(\E_n)_{n\in\N}$ \emph{locally converge} to a partition $\E$, and we write $\E_n\locto\E$, if $\E_n(i)\to\E(i)$ locally in $\R^d$ for all $i\in\{1,\ldots,N\}$, that is if $|(\asymm{\E_n(i)}{\E(i)})\cap B_R|\to0$ for all $R>0$ and for all $i\in\{1,\ldots,N\}$.

\begin{definition}[Locally isoperimetric partition] \label{def:lip}
A partition $\E_0=(\E_0(1),\ldots,\E_0(N))$ of $\R^d$ is a \emph{locally isoperimetric partition} if for every $R>0$
\begin{equation} \label{eq:lip}
\per(\E_0;B_R) \leq \per(\E;B_R)
\end{equation}
whenever $\E=(\E(1),\ldots,\E(N))$ is a partition satisfying
\begin{equation}\label{eq:lip2}
|\E(i)|=|\E_0(i)| \qquad\text{and}\qquad
\asymm{\E(i)}{\E_0(i)} \ssubset B_R
\qquad\text{for all }i\in\{1,\ldots,N\}.
\end{equation}
A locally isoperimetric partition $\E_0$ is said to be \emph{uniquely-minimizing} if the following property holds: whenever $\E$ is a $N$-partition satisfying \eqref{eq:lip2} for some $R>0$, equality in \eqref{eq:lip} implies the existence of an isometry $T:\R^d\to\R^d$ such that $\E(i)=T(\E_0(i))$ for all $i\in\{1,\ldots,N\}$.
\end{definition}

Notice that at least one region of a partition $\E$ must have infinite Lebesgue measure. As a particular case, $(N+1)$-partitions such that all regions have finite measure except one (the \emph{exterior region}) are usually referred to as $N$-clusters; we refer to \cite[Part~IV]{Maggi} for a presentation of the key ideas about existence and regularity of minimizing clusters. In this paper we are mostly concerned with partitions with at least two regions with infinite measure (in short, \emph{infinite regions}). In this case the perimeter in the full space is necessarily infinite, so that the minimality condition \eqref{eq:lip} has to be formulated locally. It is convenient to introduce a notation for the indices of the regions with finite measure: for a $N$-partition $\E$ we set
\begin{equation} \label{eq:finite_regions}
I_{\E} \defeq \bigl\{ i\in\{1,\ldots,N\} \,:\, |\E(i)|<\infty\bigr\}.
\end{equation}

The basic regularity properties of locally isoperimetric partitions are given in \cite[Theorem~2.4]{NovPaoTor}. \emph{Planar} locally isoperimetric partitions have a rigid structure that we recall from \cite[Theorem~4.1 and Theorem~4.2]{NovPaoTor}.

\begin{theorem}[Structure of planar locally isoperimetric partitions] \label{thm:regularity}
Let $d=2$ and let $\E_0$ be a locally isoperimetric partition in $\R^2$. Then $\partial\E_0$ is connected and there exist a finite family of points $\{p_i\}_{i\in I}$ (vertices) and a finite family $\{\gamma_j\}_{j\in J}$ of closed curves with boundary such that
$$
\partial\E_0 = \bigcup_{j\in J}\gamma_j, \qquad \partial^*\E_0 = \bigcup_{j\in J}\mathrm{int}(\gamma_j), \qquad \Sigma(\E_0) = \bigcup_{j\in J}\mathrm{bd(\gamma_j)}=\bigcup_{i\in I}\{p_i\},
$$
where $\mathrm{int}(\gamma)$ and $\mathrm{bd}(\gamma)$ denote the interior and the boundary points of the curve $\gamma$ respectively. Moreover:
\begin{enumerate}
\item each vertex $p_i$ is a boundary point of exactly three of the curves $\{\gamma_j\}_{j\in J}$, forming 120 degree angles at $p_i$,
\item each curve $\gamma_j$ is either a circular arc, a segment, or a half-line,
\item the three signed curvatures of the arcs meeting in a vertex have zero sum,
\item all the regions with finite area are bounded,
\item there are at most three regions with infinite area. If there are two infinite regions, the interface between them coincides with a straight line outside a sufficiently large ball; if there are three infinite regions, the interfaces between them coincide, outside a sufficiently large ball, with three half-lines whose prolongations define angles of 120 degrees with each other (but not necessarily passing through a single point).
\end{enumerate}
\end{theorem}

Given a uniquely-minimizing locally isoperimetric partition $\E_0=(\E_0(1),\ldots,\E_0(N))$, we define for $R>0$ the class of volume-constrained competitors obtained by perturbing $\E_0$ in a ball of radius $R$ (up to isometries):
\begin{multline}\label{eq:comp}
\comp_R(\E_0) \defeq
\Bigl\{ \E=(\E(1),\ldots,\E(N)) \,:\, |\E(i)|=|\E_0(i)| \,\text{ and }\, \asymm{T(\E(i))}{\E_0(i)}\ssubset B_R(x_0) \\
\text{ for all }i\in\{1,\ldots,N\}, \text{ for some }T:\R^d\to\R^d\text{ isometry and } x_0\in\R^d\Bigr\}.
\end{multline}
We set $\comp(\E_0)\defeq \bigcup_{R>0}\comp_R(\E_0)$. We measure the distance of a partition $\E\in\comp(\E_0)$ from $\E_0$ by the quantity
\begin{equation} \label{eq:asymm}
\Asymm{\E}{\E_0} \defeq \inf \Bigl\{ \dis{\E}{T(\E_0)} \,:\, T:\R^d\to\R^d \text{ isometry} \Bigr\},
\end{equation}
where $T(\E_0)$ is the partition defined by $T(\E_0)(i)\defeq T(\E_0(i))$, and for every two partitions $\E=(\E(1),\ldots,\E(N))$ and $\F=(\F(1),\ldots,\F(N))$ we have set
\begin{equation} \label{eq:distance}
\dis{\E}{\F}\defeq\frac12\sum_{i=1}^N |\asymm{\E(i)}{\F(i)}|.
\end{equation}
For a partition $\E\in\comp(\E_0)$ we introduce its \emph{perimeter deficit} (with respect to $\E_0$) as
\begin{equation} \label{eq:deficit}
\defi{\E}{\E_0} \defeq \per(T(\E);B_R(x_0)) - \per(\E_0;B_R(x_0)),
\end{equation}
where $R>0$ is any radius such that $\E\in\comp_R(\E_0)$, and $T$ and $x_0$ are as in \eqref{eq:comp}.

A natural question is, then, whether it is possible to strengthen the minimality inequality \eqref{eq:lip}, which can be rephrased as $\defi{\E}{\E_0}\geq0$ for every $\E\in\comp(\E_0)$, in a \emph{quantitative form}, namely whether it is possible to find a constant $\kappa>0$ such that
\begin{equation} \label{eq:quantitative}
\defi{\E}{\E_0} \geq \kappa \Asymm{\E}{\E_0}^2
\end{equation}
for all partitions $\E\in\comp(\E_0)$. In the case $N=2$ (i.e., one region with finite volume and one infinite region, with $\E_0=(B_r,\R^d\setminus B_r)$), the stability inequality \eqref{eq:quantitative} reduces to the celebrated \emph{quantitative isoperimetric inequality} \cite{FusMagPra08}. For minimizing clusters, to the best of our knowledge the only stability inequality available is for the planar \emph{standard double bubble} (two regions with finite area and one infinite region in $\R^2$), proved by Cicalese, Leonardi and Maggi in \cite{CicLeoMag17}.

For partitions with at least two infinite regions, however, one quickly realizes that an estimate of the form \eqref{eq:quantitative} cannot hold in the full class $\comp(\E_0)$: indeed, one can construct perturbations $\E_R$ of $\E_0$ in large balls $B_R$, $R\gg 1$, such that $\sup_R\defi{\E_R}{\E_0}<\infty$ and $\Asymm{\E_R}{\E_0}\to\infty$ as $R\to+\infty$. In dimension $d=2$ it is enough construct $\E_R$ by replacing a portion of length $R$ of one of the infinite boundaries of $\E_0$ (which are half-lines by Theorem~\ref{thm:regularity}) by another segment of the same length $R$, parallel to and at distance 1 from the first, and connect its endpoints to the rest of the boundary by two segments of length 1 each (see Figure~\ref{fig:new}). Then $\defi{\E_R}{\E_0}=2$, but $\Asymm{\E_R}{\E_0}=R$.

For this reason we restrict to the class $\comp_R(\E_0)$, that is, we impose an upper bound on the diameter of the symmetric difference of $\E$ and $\E_0$ (up to isometries), and we wish to prove the stability inequality \eqref{eq:quantitative} in $\comp_R(\E_0)$, with a constant $\kappa$ depending also on $R$.

\begin{figure}[h]
\begin{tikzpicture}[scale=1.2,line cap=round,line join=round]
	\draw[very thick] (-{sin(120)},0) arc (270-120:120-90:1);
	\draw[very thick] ({sin(120)},0) arc (450-120:90+120:1);	
	\draw[very thick] (-1.5,0) -- (-{sin(120)},0);
	\draw[very thick,dashed] (-2,0) -- (-1.5,0);	
	\draw[very thick] ({sin(120)},0) -- (1.2,0);
	\draw[thick, dashed] (1.2,0) -- (5,0);	
	\draw[very thick] (5,0) -- (6,0);
	\draw[very thick,dashed] (6,0) -- (6.5,0);	
	\draw[very thick] (1.2,0.2) -- (5,0.2);	
	\draw[very thick] (1.2,0) -- (1.2,0.2);		
	\draw[very thick] (5,0) -- (5,0.2);			
	\node[below] at (3,0) {\footnotesize $R$};
	\node[right] at (5,0.13) {\footnotesize $1$};	
\end{tikzpicture}    
\caption{Two-dimensional construction showing that the constant $\kappa$ in \eqref{eq:quantitative} should depend on the diameter $R$ of the perturbation.}\label{fig:new}
\end{figure}


\subsection{Main result: stability of the planar standard lens} \label{subsec:lens}

We now leave the general setting considered in the previous subsection and we consider \emph{planar} partitions ($d=2$) into $N=3$ regions, one of which with finite area and the remaining two with infinite measure; that is, we consider partitions $\E=(\E(1),\E(2),\E(3))$ of $\R^2$ such that $|\E(1)|=m$, where $m>0$ is a fixed parameter, and $|\E(2)|=|\E(3)|=+\infty$. In this case the perimeter of a partition $\E$ in a ball $B_R$ is simply given by
\begin{equation} \label{eq:perimeter-lens}
\per(\E;B_R) = \per(\E(1);B_R) + \Hone(\E(2,3)\cap B_R).
\end{equation} 
It has been proved in \cite{AlaBroVri} (see also \cite{NovPaoTor}) that in this case the only locally minimizing partition is given by the \emph{standard lens partition} defined below, see Figure~\ref{fig:lens}. This result has been extended to general dimension $d\geq2$ and to possibly weighted perimeter in \cite{BroNov}.

\begin{definition}[Standard lens partition] \label{def:lens}
Let $m>0$. The \emph{standard lens} with area $m$ is the set
\begin{equation} \label{eq:lens}
\lens \defeq \Bigl\{ (x,y)\in\R^2 \,:\, |x|<\fr r_{m}, \; |y| < \sqrt{r_{m}^2-x^2}- {\textstyle\frac12} r_{m} \Bigr\},
\end{equation}
where the radius $r_{m}>0$ is such that $|\lens|=m$, that is,
\begin{equation} \label{eq:lens-radius}
r_{m} = \sqrt{m}\biggl(\frac23\pi-\frac{\sqrt{3}}{2}\biggr)^{-1/2}.
\end{equation}
The \emph{standard lens partition} with area $m$ is the $3$-partition of $\R^2$ given by
\begin{equation}\label{eq:lens-cluster}
\lensc \defeq (\lens, H^+\setminus\lens, H^-\setminus\lens), 
\end{equation}
where $H^+\defeq\{ (x,y)\in\R^2 \,:\, y>0 \}$, $H^-\defeq\{ (x,y)\in\R^2 \,:\, y<0 \}$.
\end{definition}

\begin{figure}[t]
\begin{tikzpicture}[scale=3.5,line cap=round,line join=round]
	\fill[blue!15!white] (-{sin(120)},0) arc (270-120:120-90:1) arc (450-120:90+120:1);
    \fill[green!30!white] (0,{cos(120)}) -- (0,{cos(120)+0.15}) arc (90:120-90:0.15) -- (0,{cos(120)});
    \fill[green!30!white] ({sin(120)},0) -- ({sin(120)+0.1},0) arc (0:120:0.1) -- ({cos(120)/2+sin(120)},{sin(120)/2});    
	\draw[thick] (-{sin(120)},0) arc (270-120:120-90:1);
	\draw[thick] ({sin(120)},0) arc (450-120:90+120:1);	
	\draw[thick] (-1.5,0) -- (-{sin(120)},0);
	\draw[thick] ({sin(120)},0) -- (1.5,0);
	\draw [dashed] (-{sin(120)},0) -- ({sin(120)},0);
	\draw [fill=black] (-{sin(120)},0) circle (0.5pt);
	\draw [fill=black] ({sin(120)},0) circle (0.5pt);	
	\node [below left] at (-{sin(120)},0) {$p_1$};
	\node [below right] at ({sin(120)},0) {$p_2$};
	\draw [fill=black] (0,0) circle (0.5pt);
	\node [below right] at (0,0) {$O$};
	\draw [fill=black] (0,{cos(120)}) circle (0.5pt);
	\draw [dashed] (-{sin(120)},0) -- (0,{cos(120)});
	\draw [dashed] ({sin(120)},0) -- (0,{cos(120)});
	\draw [dashed] (0,0) -- (0,{cos(120)});
	\node at (0,0.2) {$\lensc(1)$};	
	\node at (-1,0.3) {$\lensc(2)$};
	\node at (-1,-0.3) {$\lensc(3)$};
	\node at (0.35,-0.2) {$r_{m}$};
    \draw (0,{cos(120)+0.15}) arc (90:120-90:0.15);
    \node at (0.1,{cos(120)+0.18}) {\small $\frac{\pi}{3}$};
    \draw ({sin(120)+0.1},0) arc (0:120:0.1);
	\draw [dashed] ({sin(120)},0) -- ({cos(120)/2+sin(120)},{sin(120)/2});    
    \node at ({sin(120)+0.12},0.13) {\small $\frac23\pi$}; 
\end{tikzpicture}
\caption{The standard lens partition $\lensc$ as in Definition~\ref{def:lens}.}\label{fig:lens}
\end{figure}

The boundary of $\lens$ is made of two symmetric circular arcs meeting at the two points $p_1=(-\fr r_{m},0)$, $p_2=(\fr r_{m},0)$ on the $x$-axis, and forming an angle $\frac23\pi$ with the $x$-axis. The two circular arc have radius $r_{m}$ and subtend an angle $\frac23\pi$ at the center. The interface between the regions $\lensc(2)$ and $\lensc(3)$ of the lens partition is flat and is given by
\begin{equation*}
\lensc(2,3) = \bigl\{ (x,0)\in\R^2 \,:\, |x| > \fr r_{m}\bigr\}.
\end{equation*}
By \cite[Theorem~1.9]{AlaBroVri} and \cite[Theorem~2.9]{BroNov}, the lens partition $\lensc$ is a uniquely-minimizing locally isoperimetric partition, in the sense of Definition~\ref{def:lip}: $\defi{\E}{\lensc}\geq0$ for any other partition $\E\in\comp(\lensc)$. Our main result is the following sharp stability inequality for the standard lens, which strengthens the local minimality of the lens partition in a quantitative form. The proof of the theorem is given at the end of Section~\ref{sec:step1}.

\begin{theorem}[Stability of the standard lens] \label{thm:stability-lens}
Let $m>0$. For every $R>0$ there exists a constant $\kappa_{m,R}>0$, depending on $m$ and $R$, such that
\begin{equation} \label{eq:stability-lens}
\defi{\E}{\lensc} \geq \kappa_{m,R}\Asymm{\E}{\lensc}^2 \qquad\text{for all }\E\in\comp_R(\lensc).
\end{equation}
\end{theorem}


\section{Selection principle and improved convergence} \label{sec:step2}

In this section we set up the general strategy for the proof of the stability of planar locally isoperimetric partitions, following the approach via \emph{improved convergence} due to Cicalese, Leonardi and Maggi \cite{CicLeoMag16}, and in turn based on the \emph{selection principle} devised by Cicalese and Leonardi \cite{CicLeo12}.

We point out that in this section we do not make use of the specific geometry of the lens partition. The main results of this part (Theorem~\ref{thm:small-asymm}, Theorem~\ref{thm:selection-principle}, and Theorem~\ref{thm:improved-conv}) are indeed valid for any uniquely-minimizing locally isoperimetric partition in $\R^2$, and might be instrumental in proving the stability of different planar locally isoperimetric partitions, other than the standard lens. Notice, however, the restriction to two dimensions, which we require in order to exploit the structure of the boundaries among infinite regions in $\R^2$ (see Theorem~\ref{thm:regularity}): indeed we have to formulate the selection principle including a Dirichlet boundary condition outside of a large ball. For \emph{clusters} (i.e. partitions with a single infinite region) these results have been proved in \cite[Appendix~A]{CicLeoMag17} in any dimension, however we focus here on the case of at least two infinite regions.

For the rest of this section we work in the general setting introduced in Section~\ref{subsec:def} and in dimension $d=2$.


\subsection{Selection principle} \label{subsec:select-principle}
We let $\E_0=(\E_0(1),\ldots,\E_0(N))$ be a uniquely-minimizing locally isoperimetric partition in $\R^2$, according to Definition~\ref{def:lip}. For $R>0$ we define the quantity
\begin{multline} \label{eq:kappa}
    \kappa_R(\E_0) \defeq \inf\biggl\{ \liminf_{k\to\infty} \frac{\defi{\E_k}{\E_0}}{\Asymm{\E_k}{\E_0}^2} \;:\; (\E_k)_{k\in\N}\subset\comp_R(\E_0), \, \Asymm{\E_k}{\E_0}>0 \text{ for all }k\in\N,\\
    \lim_{k\to\infty}\dis{\E_k}{\E_0}=0 \biggr\}
\end{multline}
where the deficit $\defi{\cdot}{\E_0}$, the distance $\Asymm{\cdot}{\E_0}$, and the class of competitors $\comp_R(\E_0)$ are defined in \eqref{eq:deficit}, \eqref{eq:asymm}, and \eqref{eq:comp} respectively.

\begin{remark} \label{rmk:sharpness}
If $\E_0$ has at least two infinite regions, it is easy to see that $\kappa_R(\E_0)$ is finite: indeed, since there is no need to preserve the volume of the infinite regions, it is enough to slightly perturb one of the flat boundaries between two infinite regions (for instance, removing a segment and replacing it by two segments forming a sawtooth of small height $t>0$) and construct partitions $\E_t$ such that the ratio $\defi{\E_t}{\E_0}/\Asymm{\E_t}{\E_0}^2$ remains bounded as $t\to0$.

From this construction it also follows that the quadratic decay on the right-hand side of \eqref{eq:quantitative} is sharp, in the following sense (see also \cite[Remark~1.2]{CicLeoMag16}): if $\defi{\E}{\E_0} \geq \phi(\Asymm{\E}{\E_0})$ for some function $\phi:[0,\infty)\to[0,\infty)$ and for all $\E\in\comp_R(\E_0)$, then there exist $C>0$ and $t_0>0$ such that $\phi(t)\leq Ct^2$ for all $t\in(0,t_0)$.
\end{remark}

In the following theorem we show that the proof of the stability inequality for $\E_0$ can be reduced to the case of partitions $\E$ with small distance $\Asymm{\E}{\E_0}$ from $\E_0$.

\begin{theorem} \label{thm:small-asymm}
Let $\E_0=(\E_0(1),\ldots,\E_0(N))$ be a uniquely-minimizing locally isoperimetric partition in $\R^2$, and let $R>0$. Then for every $\e>0$ there exists $\delta>0$ such that for every $\E\in\comp_R(\E_0)$, if $\defi{\E}{\E_0}<\delta$ then $\Asymm{\E}{\E_0}<\e$.
Furthermore, the condition
\begin{equation} \label{eq:kappa2}
\kappa_R(\E_0)>0
\end{equation}
is equivalent to the existence of a constant $\kappa_{R,\E_0}>0$, depending on $R$ and $\E_0$, such that
\begin{equation} \label{eq:kappa3}
    \defi{\E}{\E_0} \geq \kappa_{R,\E_0}\Asymm{\E}{\E_0}^2 \qquad\text{for all }\E\in\comp_R(\E_0).
\end{equation}
\end{theorem}

\begin{proof}
We argue by contradiction and we assume that, for some $R>0$, there exist $\e_*>0$ and a sequence of partitions $(\E_k)_{k\in\N}\subset\comp_{R}(\E_0)$ such that
\begin{equation} \label{proof_small-asymm_1}
\lim_{k\to\infty}\defi{\E_k}{\E_0}=0, \qquad \Asymm{\E_k}{\E_0}\geq \e_*.
\end{equation}
By definition of $\comp_{R}(\E_0)$ in \eqref{eq:comp}, and up to replacing each partition $\E_k$ by an isometric copy, we can find points $x_k\in\R^2$ such that $\asymm{\E_k(i)}{\E_0(i)}\ssubset B_R(x_k)$.

We consider first the case $\sup_k|x_k|<\infty$, so that $x_k\to x_0\in\R^2$ up to a (not relabeled) subsequence and $B_R(x_k)\subset B_{R+1}(x_0)$ for all $k$ large enough. For every $i\in\{1,\ldots,N\}$ the sets $\E_k(i)\cap B_{R+1}(x_0)$ have uniformly bounded perimeter in $B_{R+1}(x_0)$, and by standard compactness results (see \cite[Theorem~12.26]{Maggi}) we can assume that, up to further subsequences, $\E_k(i)\cap B_{R+1}(x_0)\to F_i$ as $k\to\infty$, for some set of finite perimeter $F_i\subset B_{R+1}(x_0)$. Since $\asymm{\E_k(i)}{\E_0(i)}\ssubset B_R(x_k)$, we also have that $F_i$ coincides with $\E_0(i)$ in a uniform neighbourhood of $\partial B_{R+1}(x_0)$. We then define the partition $\E_\infty$ with regions $\E_\infty(i)\defeq F_i\cup (\E_0(i)\setminus B_{R+1}(x_0))$, and by construction it is immediate to check that $\E_\infty\in\comp_{R+1}(\E_0)$. Furthermore by lower semicontinuity of the perimeter we have that
\begin{equation*}
\defi{\E_\infty}{\E_0} \leq\liminf_{k\to\infty} \defi{\E_k}{\E_0} =0, \qquad 
\Asymm{\E_\infty}{\E_0} = \lim_{k\to\infty}\Asymm{\E_k}{\E_0}\geq \e_*>0.
\end{equation*}
These two properties contradict the assumption that $\E_0$ is uniquely-minimizing.

If instead $\sup_k|x_k|=\infty$, recalling that the regions of $\E_0$ with finite area are bounded (Theorem~\ref{thm:regularity}), we have $\E_0(i)\cap B_R(x_k)=\emptyset$ for all $i\in I_{\E_0}$ and for all $k$ sufficiently large, where $I_{\E_0}$ is the set of the indices corresponding to the finite regions as in \eqref{eq:finite_regions}.
Hence $B_R(x_k)$ has nonempty intersection only with some of the infinite regions of $\E_0$. By the structure of the interfaces among the infinite regions in Theorem~\ref{thm:regularity}, we have that for all $k$ sufficiently large the interface $\partial\E_0\cap B_R(x_k)$ is either empty, or a segment. We can then find new centers $y_k\in\R^2$, with $\sup_k|y_k|<\infty$, so that $\partial\E_0\cap B_R(y_k)$ coincides with a translation of $\partial\E_0\cap B_R(x_k)$, and ``copy and paste'' $\E_k\cap B_R(x_k)$ into $B_R(y_k)$. This way we obtain a new sequence of partitions which satisfy the same properties as $\E_k$ and are perturbations of $\E_0$ inside $B_R(y_k)$. Since the new centers $y_k$ are uniformly bounded, the same compactness argument as in the previous case allows to conclude by contradiction. This completes the proof of the first part of the statement.

Concerning the equivalence between \eqref{eq:kappa2} and \eqref{eq:kappa3}, it is immediate to see that \eqref{eq:kappa3} implies \eqref{eq:kappa2}. Conversely, assume that \eqref{eq:kappa2} holds, and by contradiction that \eqref{eq:kappa3} fails, that is, there exists a sequence $(\E_k)_{k\in\N}\subset\comp_{R}(\E_0)$ with $\Asymm{\E_k}{\E_0}>0$ such that
\begin{equation} \label{proof_small-asymm_2}
    \frac{\defi{\E_k}{\E_0}}{\Asymm{\E_k}{\E_0}^2} \to 0 \qquad \text{as }k\to\infty.
\end{equation}
We first observe that $\Asymm{\E_k}{\E_0}\to0$ as $k\to\infty$. Indeed, if not then $\Asymm{\E_k}{\E_0}\geq\e$ for some $\e>0$ and in turn, by the first part of the statement, also $\defi{\E_k}{\E_0}\geq\delta>0$. Moreover, since $\E_k\in\comp_{R}(\E_0)$ we also have $\Asymm{\E_k}{\E_0}\leq|B_R|$. These inequalities contradict \eqref{proof_small-asymm_2}, showing that $\Asymm{\E_k}{\E_0}\to0$.

Let now $T:\R^2\to\R^2$ be an isometry such that $\Asymm{\E_k}{\E_0}=\dis{T(\E_k)}{\E_0}$. The sequence of partitions $\F_k\defeq T(\E_k)$ is such that $\F_k\in\comp_R(\E_0)$ and $\dis{\F_k}{\E_0}\to0$, and is therefore admissible in the definition of $\kappa_R(\E_0)$ in \eqref{eq:kappa}. It follows
\begin{equation*}
    0<\kappa_R(\E_0) \leq \liminf_{k\to\infty}\frac{\defi{\F_k}{\E_0}}{\Asymm{\F_k}{\E_0}^2} = \liminf_{k\to\infty}\frac{\defi{\E_k}{\E_0}}{\Asymm{\E_k}{\E_0}^2} = 0,
\end{equation*}
which is a contradiction.
\end{proof}

Thanks to Theorem~\ref{thm:small-asymm}, our final goal will be to show the strict inequality $\kappa_R(\lensc)>0$ for the standard lens partition. We next show in Theorem~\ref{thm:selection-principle} the existence of a recovery sequence for $\kappa_R(\E_0)$ made of \emph{$(\Lambda,r_0)$-minimizing partitions} (with uniform constants), according to the following definition.

\begin{definition} \label{def:quasi-min}
Given $\Lambda>0$ and $r_0>0$, a partition $\E=(\E(1),\ldots,\E(N))$ is said to be a \emph{$(\Lambda,r_0)$-minimizing partition} if
\begin{equation} \label{eq:quasi-min}
\per(\E;B_{r_0}(x)) \leq \per(\F;B_{r_0}(x)) + \Lambda \dis{\E}{\F}
\end{equation}
whenever $x\in\R^2$ and $\F$ is a partition such that $\asymm{\F(i)}{\E(i)}\ssubset B_{r_0}(x)$.
\end{definition}

In the proof of Theorem~\ref{thm:selection-principle} we need the following construction of a suitable neighbourhood $O_R$ of a ball $B_R$, depending on the structure of the infinite regions of $\E_0$ (see Figure~\ref{fig:eye}).

\begin{definition} \label{def:eye}
Let $\E_0=(\E_0(1),\ldots,\E_0(N))$ be a uniquely-minimizing locally isoperimetric partition in $\R^2$ with $k$ regions with infinite area, where $k$ is either 2 or 3.
\begin{enumerate}
    \item Let $R\gg1$ be such that $\partial\E_0\setminus B_R$ is the union of $k$ half-lines, according to Theorem~\ref{thm:regularity}.
    By possibly taking a larger $R$, we can assume that each half-line intersects $\partial B_R$ with an angle close to $\frac\pi2$ (say, between $\frac\pi3$ and $\frac{2\pi}{3}$).
    \item Let $L$ be any of these $k$ half-lines, meeting $\partial B_R$ at a point $p$. Let $\partial B_R \cap \partial B_{1/2}(p)=\{q_1,q_2\}$, and let $(L\cap\partial B_{1/2}(p))\setminus B_R=\{q_0\}$. For $i=1,2$, connect the point $q_i$ to $q_0$ by a smooth curve $\Gamma_i\subset \overline{B_{1/2}(p)}\setminus B_R$, meeting $\partial B_R$ at $q_i$ and $L$ at $q_0$ in a $C^2$-way, and intersecting $L$ only at $q_0$. Let $E_L$ be the region enclosed by the curves $\Gamma_1$, $\Gamma_2$, and $\partial B_{R}\cap B_{1/2}(p)$.
    \item Let $O_R$ be the open set obtained by the union of $B_R$ with each of the $k$ sets $E_L$ constructed in the previous point.
\end{enumerate}
Notice that $B_R\subset O_R \ssubset B_{R+1}$, and the boundary of $O_R$ is a curve of class $C^2$, except for $k$ cusp points at the intersection $\partial O_R\cap\partial\E_0$.
\end{definition}

\begin{figure}[ht]
\begin{tikzpicture}[scale=1.6,line cap=round,line join=round]
    \clip(-2,-1.7) rectangle (2,1.8);
    \draw [fill=black] (0,0) circle (0.5pt);
    \draw [fill=black] (1.5,0) circle (0.5pt);	
	\node [below] at (1.5,0) {$q_0$};
    \draw [fill=black] (1,0) circle (0.5pt);	
	\node [below left] at (1,0) {$p$};
    \draw [fill=black] ({cos(28.96)},{sin(28.96)}) circle (0.5pt);	
	\node [left] at ({cos(31.7)},{sin(31.7)}) {$q_1$};
    \draw [fill=black] ({cos(28.96)},{-sin(28.96)}) circle (0.5pt);	
	\node [left] at ({cos(31.7)},{-sin(31.7)}) {$q_2$};
    \draw[dashed] (0,0) circle (1);
    \draw[dashed] (1,0) circle (0.5);
    \draw[dashed] (-1.5,0) -- (1.5,0);
    \draw[thick] (-2,0) -- (-1.5,0);    
    \draw[thick] (1.5,0) -- (2,0);       
    \draw[thick] ({cos(31.7)},{sin(31.7)}) arc (31.7:148.3:1);
    \draw[thick] ({cos(211.7)},{sin(211.7)}) arc (211.7:328.3:1);
	\draw[scale=1,domain=0.8508:1.5,smooth,variable=\x,thick] plot ({\x},{1.2468*(\x-1.5)*(\x-1.5)});    
	\draw[scale=1,domain=0.8508:1.5,smooth,variable=\x,thick] plot ({\x},{-1.2468*(\x-1.5)*(\x-1.5)});    
	\draw[scale=1,domain=-1.5:-0.8508,smooth,variable=\x,thick] plot ({\x},{1.2468*(\x+1.5)*(\x+1.5)});    
	\draw[scale=1,domain=-1.5:-0.8508,smooth,variable=\x,thick] plot ({\x},{-1.2468*(\x+1.5)*(\x+1.5)});    			
\end{tikzpicture}
\qquad\qquad
\begin{tikzpicture}[scale=1.5,line cap=round,line join=round,declare function={myfun(\x)  = 1.2468*(\x-1.5)*(\x-1.5);}]
    \draw [fill=black] (0,0) circle (0.5pt);
    \draw [fill=black] (1.5,0) circle (0.5pt);	
	\node [below] at (1.5,0) {$q_0$};
    \draw [fill=black] (1,0) circle (0.5pt);	
	\node [below left] at (1,0) {$p$};
    \draw [fill=black] ({cos(28.96)},{sin(28.96)}) circle (0.5pt);	
	\node [left] at ({cos(31.7)},{sin(31.7)}) {$q_1$};
    \draw [fill=black] ({cos(28.96)},{-sin(28.96)}) circle (0.5pt);	
	\node [left] at ({cos(31.7)},{-sin(31.7)}) {$q_2$};
    \draw[dashed] (0,0) circle (1);
    \draw[dashed] (1,0) circle (0.5);
    \draw[dashed] (0,0) -- (1.5,0);
    \draw[thick] (1.5,0) -- (2,0);
    \draw[dashed] (0,0) -- ({1.5*cos(120)},{1.5*sin(120)});
    \draw[thick] ({1.5*cos(120)},{1.5*sin(120)}) -- ({2*cos(120)},{2*sin(120)});    
    \draw[dashed] (0,0) -- ({1.5*cos(240)},{1.5*sin(240)});             
    \draw[thick] ({1.5*cos(240)},{1.5*sin(240)}) -- ({2*cos(240)},{2*sin(240)});        
    \draw[thick] ({cos(31.7)},{sin(31.7)}) arc (31.7:88.3:1);
    \draw[thick] ({cos(151.7)},{sin(151.7)}) arc (151.7:208.3:1);    
    \draw[thick] ({cos(271.7)},{sin(271.7)}) arc (271.7:328.3:1); 
	\draw[scale=1,domain=0.8508:1.5,smooth,variable=\x,thick] plot ({\x},{myfun(\x)});    	  
	\draw[scale=1,domain=0.8508:1.5,smooth,variable=\x,thick] plot ({\x},{-myfun(\x)});    	  	
	\draw[scale=1,domain=0.8508:1.5,smooth,variable=\x,thick] plot ({cos(120)*\x-sin(120)*myfun(\x)},{sin(120)*\x + cos(120)*myfun(\x)});
	\draw[scale=1,domain=0.8508:1.5,smooth,variable=\x,thick] plot ({cos(120)*\x+sin(120)*myfun(\x)},{sin(120)*\x - cos(120)*myfun(\x)});  
	\draw[scale=1,domain=0.8508:1.5,smooth,variable=\x,thick] plot ({cos(240)*\x-sin(240)*myfun(\x)},{sin(240)*\x + cos(240)*myfun(\x)});
	\draw[scale=1,domain=0.8508:1.5,smooth,variable=\x,thick] plot ({cos(240)*\x+sin(240)*myfun(\x)},{sin(240)*\x - cos(240)*myfun(\x)});  
\end{tikzpicture}
\caption{The set $O_R$ constructed in Definition~\ref{def:eye}, depending whether $\E_0$ has two infinite regions (left) or three infinite regions (right).}\label{fig:eye}
\end{figure}

\begin{theorem} \label{thm:selection-principle}
Let $\E_0=(\E_0(1),\ldots,\E_0(N))$ be a uniquely-minimizing locally isoperimetric partition in $\R^2$ with at least two infinite regions, and let $\Rba>0$.

There exist a radius $R\geq \Rba$, positive constants $\Lambda>0$ and $r_0>0$ (all depending only on $\E_0$ and $\Rba$), and a sequence $(\E_k)_{k\in\N}\subset\comp_{R+1}(\E_0)$ of $(\Lambda,r_0)$-minimizing partitions, such that $\asymm{\E_k(i)}{\E_0(i)}\ssubset B_{R+1}$ for all $i$, and
\begin{equation} \label{eq:selection-principle}
    \Asymm{\E_k}{\E_0}>0 \quad\text{for all }k\in\N, \qquad \lim_{k\to\infty}\Asymm{\E_k}{\E_0}=0, \qquad \lim_{k\to\infty}\frac{\defi{\E_k}{\E_0}}{\Asymm{\E_k}{\E_0}^2} = \kappa_{\Rba}(\E_0).
\end{equation}
\end{theorem}

\begin{proof}
We divide the proof into three steps.

\medskip\noindent \emph{Step 1: localization in a large ball.}
Let $(\F_k)_{k\in\N}\subset\comp_{\Rba}(\E_0)$ be a recovery sequence for $\kappa_{\Rba}(\E_0)$ in \eqref{eq:kappa}, that is,
\begin{equation} \label{proof:selection-principle-1}
    \Asymm{\F_k}{\E_0}>0 \quad\text{for all }k\in\N, \qquad \lim_{k\to\infty}\Asymm{\F_k}{\E_0}=0, \qquad \lim_{k\to\infty}\frac{\defi{\F_k}{\E_0}}{\Asymm{\F_k}{\E_0}^2} = \kappa_{\Rba}(\E_0).
\end{equation}
By definition of the class $\comp_{\Rba}(\E_0)$, up to replacing $\F_k$ by an isometric copy, we can find points $x_k\in\R^2$ such that $\asymm{\F_k(i)}{\E_0(i)}\ssubset B_{\Rba}(x_k)$ for all $i\in\{1,\ldots,N\}$. By finiteness of $\kappa_{\Rba}(\E_0)$, it follows from \eqref{proof:selection-principle-1} that as $k\to\infty$
\begin{equation} \label{proof:selection-principle-2}
    \per(\F_k;B_{\Rba}(x_k))-\per(\E_0;B_{\Rba}(x_k)) = \defi{\F_k}{\E_0} = \kappa_{\Rba}(\E_0)\Asymm{\F_k}{\E_0}^2 + o(\Asymm{\F_k}{\E_0}^2),
\end{equation}
so that by assuming $k$ large enough we can bound
\begin{equation} \label{proof:selection-principle-2b}
    \per(\F_k;B_{\Rba}(x_k))-\per(\E_0;B_{\Rba}(x_k)) \leq \bigl(\kappa_{\Rba}(\E_0)+1\bigr)\Asymm{\F_k}{\E_0}^2.
\end{equation}

We claim that there exists $R\geq{\Rba}$, depending only on ${\Rba}$ and $\E_0$, such that we can assume
\begin{equation} \label{proof:selection-principle-3}
B_{\Rba}(x_k)\subset B_{R} \quad\text{for all $k$,}
\qquad
\E_0(i)\ssubset B_{R} \quad\text{for all $i\in I_{\E_0}$.}
\end{equation}
Indeed, since by Theorem~\ref{thm:regularity} the regions of $\E_0$ with finite area are bounded, there exists a radius $R_0>0$ depending only on $\E_0$ such that $\E_0(i)\ssubset B_{R_0}$ for all $i\in I_{\E_0}$, so that the second condition in \eqref{proof:selection-principle-3} is satisfied for all $R>R_0$. Regarding the first condition, we have that either $B_{\Rba}(x_k)\cap \E_0(i)\neq\emptyset$ for some $i\in I_{\E_0}$ (and in this case it is enough to take $R>R_0+2{\Rba}$ to guarantee that $B_{\Rba}(x_k)\subset B_{R}$), or $B_{\Rba}(x_k)\cap \E_0(i)=\emptyset$ for all $i\in I_{\E_0}$. In this second case we can use the same argument as in the proof of Theorem~\ref{thm:small-asymm} to ``copy and paste'' $\F_k\cap B_{\Rba}(x_k)$ into balls $B_{\Rba}(y_k)$ that are contained in a uniform ball $B_{R}$, so that we can assume without loss of generality that also the first condition in \eqref{proof:selection-principle-3} holds.

\medskip\noindent \emph{Step 2: construction of $\E_k$.}
Let $O_{R}$ be the set introduced in Definition~\ref{def:eye}, depending on the structure of the infinite regions of $\E_0$, see in particular Figure~\ref{fig:eye} (notice that, up to taking a larger $R$, the assumption in Definition~\ref{def:eye} is satisfied). We define $\E_k$ as a solution to the following minimum problem:
\begin{multline} \label{proof:selection-principle-4}
    \min\biggl\{ \per(\E;B_{R+1}) + \big| \Asymm{\E}{\E_0}-\Asymm{\F_k}{\E_0}\big|^{3/2} \,:\, |\E(i)|=|\E_0(i)| \text{ and } \\
	|(\asymm{\E(i)}{\E_0(i)})\setminus O_{R}|=0 \text{ for all }i\in\{1,\ldots,N\} \biggr\}
\end{multline}
whose existence follows by the direct method of the calculus of variations. Notice that, since $O_{R}\ssubset B_{R+1}$, we have $\asymm{\E_k(i)}{\E_0(i)}\ssubset B_{R+1}$, and $\E_k$ satisfies the volume constraint as well; hence $\E_k\in\comp_{R+1}(\E_0)$ for all $k$.

Furthermore, $\F_k$ is admissible as competitor in \eqref{proof:selection-principle-4}, hence by minimality of $\E_0$ and by \eqref{proof:selection-principle-2b} we have
\begin{equation} \label{proof:selection-principle-5}
    \begin{split}
    \per(\E_k;B_{R+1}) + \big| \Asymm{\E_k}{\E_0}-\Asymm{\F_k}{\E_0}\big|^{3/2}
    & \leq \per(\F_k;B_{R+1}) \\
    & \leq \per(\E_0;B_{R+1}) + \bigl(\kappa_{\Rba}(\E_0)+1\bigr)\Asymm{\F_k}{\E_0}^2.
    \end{split}
\end{equation}
Since $\per(\E_k;B_{R+1})-\per(\E_0;B_{R+1})\geq 0$ by local minimality of $\E_0$, it follows from \eqref{proof:selection-principle-5}
\begin{equation} \label{proof:selection-principle-6}
   \big| \Asymm{\E_k}{\E_0}-\Asymm{\F_k}{\E_0}\big|^{3/2} \leq \bigl(\kappa_{\Rba}(\E_0)+1\bigr)\Asymm{\F_k}{\E_0}^2.
\end{equation}
In turn, since $0<\Asymm{\F_k}{\E_0}\to0$, dividing by $\Asymm{\F_k}{\E_0}^{3/2}$ in \eqref{proof:selection-principle-6} we find
\begin{equation} \label{proof:selection-principle-7}
    \lim_{k\to\infty}\frac{\Asymm{\E_k}{\E_0}}{\Asymm{\F_k}{\E_0}}=1,
    \qquad 
    \Asymm{\E_k}{\E_0}>0 \quad\text{for all }k\in\N,
    \qquad \lim_{k\to\infty}\Asymm{\E_k}{\E_0}=0.
\end{equation}
Again by \eqref{proof:selection-principle-5} and \eqref{proof:selection-principle-2} we have
\begin{equation}\label{proof:selection-principle-8}
    \begin{split}
    	\defi{\E_k}{\E_0}
        = \per(\E_k;B_{R+1})-\per(\E_0;B_{R+1})
        & \leq \per(\F_k;B_{R+1})-\per(\E_0;B_{R+1}) \\
        & = \kappa_{\Rba}(\E_0)\Asymm{\F_k}{\E_0}^2 + o(\Asymm{\F_k}{\E_0}^2)\\
        & = \kappa_{\Rba}(\E_0)\Asymm{\E_k}{\E_0}^2 + o(\Asymm{\E_k}{\E_0}^2),
    \end{split}
\end{equation}
where the last identity follows from \eqref{proof:selection-principle-7}. In view of \eqref{proof:selection-principle-7} and \eqref{proof:selection-principle-8}, the sequence $(\E_k)_{k\in\N}$ satisfies the conditions in \eqref{eq:selection-principle}.

We further notice for later use that we can find an isometry $T:\R^2\to\R^2$ such that
\begin{equation} \label{proof:selection-principle-9a}
\lim_{k\to\infty}\dis{\E_k}{T(\E_0)}=0,
\qquad\qquad
T(\E_0(i))\subset O_{R} \quad\text{for all }i\in I_{\E_0},
\end{equation}
and
\begin{equation} \label{proof:selection-principle-9b}
|(\asymm{T(\E_0(i))}{\E_0(i)})\setminus O_{R}| = 0 \quad\text{for all }i\in\{1,\ldots,N\}.
\end{equation}
Indeed, let $T_k:\R^2\to\R^2$ be isometries such that $\dis{\E_k}{T_k(\E_0)}=\Asymm{\E_k}{\E_0}\to0$. Notice that $\sup_k|T_k(0)|<\infty$ in view of the condition $\asymm{\E_k(i)}{\E_0(i)}\ssubset B_{R+1}$, so that, up to a subsequence, $T_k$ converge to some isometry $T$. By triangle inequality the first condition in \eqref{proof:selection-principle-9a} is satisfied. The second condition holds since, if by contradiction $|T(\E_0(i))\setminus O_{R}|>0$ for some $i\in I_{\E_0}$, then we would have $|\E_k(i)\setminus O_{R}|>0$ for $k$ sufficiently large, and in turn also $|\E_0(i)\setminus O_{R}|>0$ since $\E_k$ coincides with $\E_0$ outside $O_{R}$, contradicting \eqref{proof:selection-principle-3}. For the same reason, also \eqref{proof:selection-principle-9b} holds.

\medskip\noindent \emph{Step 3: $(\Lambda,r_0)$-minimality.}
We are left to prove that $\E_k$ is a $(\Lambda,r_0)$-minimizing partition for all $k$ according to Definition~\ref{def:quasi-min}, for suitable uniform constants $\Lambda>0$, $r_0>0$. We choose $r_0>0$ given by Lemma~\ref{lem:volume-fixing} corresponding to $T(\E_0)$ (notice that $r_0$ depends ultimately only on $\E_0$ and ${\Rba}$).

Fix $k\in\N$ and let $\F=(\F(1),\ldots,\F(N))$ be a partition such that $\asymm{\E_k(i)}{\F(i)}\ssubset B_{r_0}(x)$ for some $x\in\R^2$. We shall prove the inequality \eqref{eq:quasi-min} for a suitable $\Lambda>0$.

Assume first that $B_{r_0}(x)\cap O_{R}=\emptyset$. Recalling \eqref{proof:selection-principle-3}, $\partial\E_0\cap B_{r_0}(x)$ is either empty or a segment (the boundary among two infinite regions, say $\E_0(i_1)$ and $\E_0(i_2)$). Then, since $|(\asymm{\E_k(i)}{\E_0(i)})\setminus O_{R}|=0$ for all $i$, we have
\begin{align*}
    \per(\E_k; B_{r_0}(x))
    & = \per(\E_0; B_{r_0}(x))
    = \per(\E_0(i_1); B_{r_0}(x))
    \leq \per(\F(i_1); B_{r_0}(x))
    \leq \per(\F; B_{r_0}(x)),
\end{align*}
where the first inequality follows from the fact that $\partial\E_0(i_1)\cap B_{r_0}(x)$ is either empty or a segment, and $\asymm{\F(i_1)}{\E_0(i_1)}\ssubset B_{r_0}(x)$. The previous estimate proves \eqref{eq:quasi-min} in this case (with $\Lambda=0$).

Consider next the case $B_{r_0}(x)\cap O_{R}\neq\emptyset$ and let us prove \eqref{eq:quasi-min} also in this case. By taking $r_0<\frac12$ we can assume that $B_{r_0}(x)\ssubset B_{R+1}$. We modify the partition $\F$ in order to obtain an admissible competitor for the minimum problem \eqref{proof:selection-principle-4}. We have two constraints to satisfy, and we proceed in two steps:
\begin{itemize}
    \item Let $\F'=(\F'(1),\ldots,\F'(N))$ be obtained from $\F$ by setting
    $$
    \F'(i) \defeq \bigl( \F(i)\cap O_{R} \bigr) \cup \bigl( \E_0(i) \setminus O_{R} \bigr) \qquad\text{for all }i\in\{1,\ldots,N\}.
    $$
    Notice that $|(\asymm{\F'(i)}{\E_0(i)})\setminus O_{R}|=0$ and $\asymm{\F'(i)}{\E_k(i)}\ssubset B_{r_0}(x)$.
    \item Let $\F''=(\F''(1),\ldots,\F''(N))$ be obtained from $\F'$ by apply the volume-fixing variation Lemma~\ref{lem:volume-fixing} with $\E=\E_k$, and $\G=\F'$ (that is, we set $\F''=\widetilde{\G}$ obtained by the lemma with the previous choices). The lemma can be applied since $\dis{\E_k}{T(\E_0)}<\e_0$ for $k$ sufficiently large by \eqref{proof:selection-principle-9a}, and $\asymm{\F'(i)}{\E_k(i)}\ssubset B_{r_0}(x)$.
\end{itemize}

Let $C_0$ be the constant given by Lemma~\ref{lem:volume-fixing}, and let $C_1>0$ be such that $\per(\E_k;B_{R+1})\leq C_1$ for all $k$, which exists by \eqref{proof:selection-principle-5}. By the third property in Lemma~\ref{lem:volume-fixing} the partition $\F''$ satisfies the estimate
\begin{equation} \label{proof:qua-min-0a}
\begin{split}
    |\per(\F'';B_{R+1})-\per(\F';B_{R+1})|
    & \leq C_0\per(\E_k;B_{R+1})\sum_{i\in I_{\E_0}}\big| |\F'(i)|-|\E_k(i)| \big| \\
    & \leq C_0C_1\sum_{i\in I_{\E_0}}\big| \asymm{\F'(i)}{\E_k(i)} \big| \\
    & \leq C_0C_1\sum_{i\in I_{\E_0}}\big| \asymm{\F(i)}{\E_k(i)} \big|,   
\end{split}
\end{equation}
where the last passage follows since
\begin{align} \label{proof:qua-min-0}
    \big|\asymm{\F(i)}{\E_k(i)}\big|
    \geq \big| \bigl(\asymm{\F(i)}{\E_k(i)}\bigr)\cap O_{R} \big|
    = \big| \bigl(\asymm{\F'(i)}{\E_k(i)}\bigr)\cap O_{R} \big|
    = \big| \asymm{\F'(i)}{\E_k(i)} \big|.
\end{align}
Similar to \eqref{proof:qua-min-0a}, by the fourth property in Lemma~\ref{lem:volume-fixing} we also have
\begin{equation} \label{proof:qua-min-0b}
    |\dis{\F''}{\E_k}-\dis{\F'}{\E_k}| \leq C_0C_1\sum_{i\in I_{\E_0}}\big| \asymm{\F(i)}{\E_k(i)} \big|.
\end{equation}

Notice that $|(\asymm{\F''(i)}{\E_0(i)})\setminus O_{R}|=0$ for all $i$, since $\F'$ satisfies this property and $\F''$ is obtained by perturbing $\F'$ inside $O_{R}$ (first property in Lemma~\ref{lem:volume-fixing}). Moreover, $|\F''(i)|=|\E_k(i)|=|\E_0(i)|$ for all $i$ (by the second property in Lemma~\ref{lem:volume-fixing}). Hence $\F''$ obeys both constraints in the minimum problem \eqref{proof:selection-principle-4}. Defining
\begin{equation*}
    \Delta_k \defeq \big| \Asymm{\F''}{\E_0}-\Asymm{\F_k}{\E_0}\big|^{3/2} - \big| \Asymm{\E_k}{\E_0}-\Asymm{\F_k}{\E_0}\big|^{3/2},
\end{equation*}
we have, by minimality of $\E_k$ in \eqref{proof:selection-principle-4},
\begin{align} \label{proof:quamin-1}
\per(\E_k;B_{R+1}) 
&\leq \per(\F'';B_{R+1}) + \Delta_k \nonumber\\
& \xupref{proof:qua-min-0a}{\leq} \per(\F';B_{R+1}) + C_0C_1\sum_{i\in I_{\E_0}}\big| \asymm{\F(i)}{\E_k(i)} \big| + \Delta_k \\
& = \per(\F;B_{R+1}) + \Bigl( \per(\F';B_{R+1})-\per(\F;B_{R+1}) \Bigr) + C_0C_1\sum_{i\in I_{\E_0}}\big| \asymm{\F(i)}{\E_k(i)} \big| + \Delta_k. \nonumber
\end{align}
We now estimate $\Delta_k$. By using the elementary inequality $|a^{3/2}-b^{3/2}|\leq\frac32\sqrt{\max\{a,b\}}|a-b|$ for all $a,b\geq0$, and observing that $\Asymm{\F_k}{\E_0}\to0$, $\Asymm{\E_k}{\E_0}\to0$, and $\Asymm{\F''}{\E_0}\leq|O_{R}|$, we have, for a constant $C_2$ depending on $R$,
\begin{align*}
    |\Delta_k|
    & \leq C_2 \big| \Asymm{\F''}{\E_0} - \Asymm{\E_k}{\E_0} \big|
    \leq C_2 \, \dis{\F''}{\E_k} \\
    & \xupref{proof:qua-min-0b}{\leq} C_2\,\dis{\F'}{\E_k} + C_0C_1C_2\sum_{i\in I_{\E_0}}\big| \asymm{\F(i)}{\E_k(i)} \big| \\
    & \leq \frac{C_2}{2}\sum_{i=1}^N \big| \asymm{\F(i)}{\E_k(i)} \big| + C_0C_1C_2\sum_{i\in I_{\E_0}}\big| \asymm{\F(i)}{\E_k(i)} \big|,
\end{align*}
where the last passage follows from \eqref{proof:qua-min-0}.
By inserting this estimate into \eqref{proof:quamin-1} we get, for a constant $C_3$ depending on $\E_0$ and $R$,
\begin{equation} \label{proof:quamin-2}
  \per(\E_k;B_{R+1}) \leq \per(\F;B_{R+1}) + \Bigl( \per(\F';B_{R+1})-\per(\F;B_{R+1}) \Bigr) + C_3\sum_{i=1}^N |\asymm{\E_k(i)}{\F(i)}|.
\end{equation}

It remains to estimate the change in perimeter between $\F$ and $\F'$.
In the following computations, in view of the regularity of $\E_0$ given by Theorem~\ref{thm:regularity}, we assume that the regions $\E_0(i)$ are open sets. By definition of $\F'$ we have
\begin{equation} \label{proof:qua-min-3}
\per(\F';B_{R+1}) - \per(\F;B_{R+1})
= \per(\E_0; B_{R+1}\setminus\overline{O_{R}}) + \per(\F';\partial O_{R}) - \per(\F;B_{R+1}\setminus O_{R}).
\end{equation}
To compute the perimeter of $\F'$ on $\partial O_{R}$, we recall that by \eqref{proof:selection-principle-3} the finite regions of $\E_0(i)$ are compactly contained in $O_{R}$ and that $\per(\E_0;\partial O_{R})=0$, hence we can decompose $\partial O_{R}$, up to a negligible set, into the disjoint union of the sets $\partial O_{R}\cap\E_0(i)$ for $i\in I_{\E_0}^c\defeq \{1,\ldots,N\}\setminus I_{\E_0}$ (the indices corresponding to the infinite regions). Therefore
\begin{align*}
    \per(\F';\partial O_{R})
    & = \sum_{i\in I_{\E_0}^c} \per(\F';\partial O_{R}\cap \E_0(i)) \\
    & = \sum_{i\in I_{\E_0}^c} \Hone\bigl(\partial O_{R}\cap \E_0(i)\cap \F(i)^{(0)}\bigr)
    + \sum_{i\in I_{\E_0}^c} \Hone\bigl( \{\nu_{\F(i)}=-\nu_{O_{R}}\}\cap \E_0(i) \bigr).
\end{align*}
Here $A^{(\theta)}$ denotes the set of points where $A$ has Lebesgue density $\theta\in[0,1]$, and we write $\{\nu_A=\pm\nu_B\}\defeq\{x\in\partial^*A\cap\partial^*B \,:\, \nu_A(x)=\pm\nu_B(x)\}$ for any two sets of finite perimeter $A$ and $B$. Similarly,
\begin{align*}
    \per(\F;B_{R+1} & \setminus O_{R})
    = \sum_{i\in I_{\E_0}^c} \per(\F;\E_0(i)\cap B_{R+1}\setminus O_{R})
    + \frac12\sum_{i\in I_{\E_0}^c} \per(\F;\partial^*\E_0(i)\cap B_{R+1}\setminus O_{R}) \\
    & \geq \sum_{i\in I_{\E_0}^c} \per(\F(i);\E_0(i)\cap B_{R+1}\setminus O_{R})
    + \frac12\sum_{i\in I_{\E_0}^c} \Hone\bigl( \{\nu_{\F(i)}=\nu_{\E_0(i)}\}\cap B_{R+1}\setminus\overline{O_{R}} \bigr).
\end{align*}
By inserting the last two identities into \eqref{proof:qua-min-3} we obtain
\begin{equation} \label{proof:qua-min-4}
\begin{split}
    \per & (\F'; B_{R+1}) - \per(\F;B_{R+1})
    \leq \per(\E_0; B_{R+1}\setminus\overline{O_{R}}) \\
    & + \sum_{i\in I_{\E_0}^c} \Hone\bigl(\partial O_{R}\cap \E_0(i)\cap \F(i)^{(0)}\bigr)
    + \sum_{i\in I_{\E_0}^c} \Hone\bigl( \{\nu_{\F(i)}=-\nu_{O_{R}}\}\cap \E_0(i) \bigr) \\
    & - \sum_{i\in I_{\E_0}^c} \per(\F(i);\E_0(i)\cap B_{R+1}\setminus O_{R})
    - \frac12\sum_{i\in I_{\E_0}^c} \Hone\bigl( \{\nu_{\F(i)}=\nu_{\E_0(i)}\}\cap B_{R+1}\setminus\overline{O_{R}} \bigr).
\end{split} 
\end{equation}

Define now for $i\in I_{\E_0}^c$ the sets
\begin{equation*}
    \G(i) \defeq \bigl(\E_0(i) \cap \F(i)\bigr) \setminus O_{R} .
\end{equation*}
By \cite[Theorem~16.3]{Maggi} we have
\begin{align*}
    \per(\G(i);B_{R+1})
    & = \per(\F(i); \E_0(i)\cap B_{R+1}\setminus\overline{O_{R}})
    + \per(\E_0(i); \F(i)^{(1)}\cap B_{R+1}\setminus\overline{O_R}) \\
    & \qquad + \Hone\bigl( \partial O_R \cap \E_0(i) \cap \F(i)^{(1)} \bigr) \\
    & \qquad + \Hone\bigl( \{\nu_{\F(i)}=\nu_{\E_0(i)}\}\cap B_{R+1}\setminus\overline{O_{R}} \bigr)
    + \Hone\bigl( \{\nu_{\F(i)}=-\nu_{O_{R}}\}\cap\E_0(i)\bigr).
\end{align*}
Inserting this identity into \eqref{proof:qua-min-4} we find
\begin{align*}
  \per(\F';& B_{R+1}) - \per(\F;B_{R+1})  
  \leq \per(\E_0; B_{R+1}\setminus\overline{O_{R}}) \\
    & + \sum_{i\in I_{\E_0}^c} \Hone\bigl(\partial O_{R}\cap \E_0(i)\cap \F(i)^{(0)}\bigr)
    + \sum_{i\in I_{\E_0}^c} \Hone\bigl( \{\nu_{\F(i)}=-\nu_{O_{R}}\}\cap \E_0(i) \bigr) \\
    & - \sum_{i\in I_{\E_0}^c} \per(\F(i);\E_0(i)\cap\partial O_{R})
    + \frac12\sum_{i\in I_{\E_0}^c} \Hone\bigl( \{\nu_{\F(i)}=\nu_{\E_0(i)}\}\cap B_{R+1}\setminus\overline{O_{R}} \bigr) \\
    & + \sum_{i\in I_{\E_0}^c}\per(\E_0(i); \F(i)^{(1)}\cap B_{R+1}\setminus\overline{O_R})
    + \sum_{i\in I_{\E_0}^c} \Hone\bigl( \partial O_R \cap \E_0(i) \cap \F(i)^{(1)} \bigr) \\
    & + \sum_{i\in I_{\E_0}^c}\Hone\bigl( \{\nu_{\F(i)}=-\nu_{O_{R}}\}\cap\E_0(i)\bigr)
     - \sum_{i\in I_{\E_0}^c}\per(\G(i);B_{R+1}).
\end{align*}
We can group together all the terms on $\partial O_R$, whose combination is controlled by $\Hone(\partial O_R)$:
\begin{align*}
    \Hone\bigl( & \partial O_{R}\cap \E_0(i)\cap \F(i)^{(0)}\bigr)
    + 2\Hone\bigl( \{\nu_{\F(i)}=-\nu_{O_{R}}\}\cap \E_0(i) \bigr) \\
    & \qquad
    - \per(\F(i);\E_0(i)\cap\partial O_{R})
    + \Hone\bigl( \partial O_R \cap \E_0(i) \cap \F(i)^{(1)} \bigr) \\
    & \leq \Hone\bigl(\partial O_{R}\cap \E_0(i)\cap \F(i)^{(0)}\bigr)
    + \Hone\bigl( \partial O_R \cap \E_0(i) \cap \F(i)^{(1)} \bigr)
    + \Hone\bigl( \partial O_R \cap \E_0(i) \cap \partial^*\F(i) \bigr) \\
    & = \Hone(\partial O_R\cap\E_0(i)).
\end{align*}
Therefore
\begin{align*}
  \per(\F'; B_{R+1}) - \per(\F;B_{R+1})    
  & \leq \frac12\sum_{i\in I_{\E_0}^c}\per(\E_0(i);B_{R+1}\setminus O_R)
  + \sum_{i\in I_{\E_0}^c}\Hone\bigl(\partial O_R \cap \E_0(i)\bigr) \\
  & \qquad + \frac12\sum_{i\in I_{\E_0}^c} \Hone\bigl( \{\nu_{\F(i)}=\nu_{\E_0(i)}\}\cap B_{R+1}\setminus\overline{O_{R}} \bigr) \\
  & \qquad + \sum_{i\in I_{\E_0}^c}\per(\E_0(i); \F(i)^{(1)}\cap B_{R+1}\setminus\overline{O_R})
  - \sum_{i\in I_{\E_0}^c}\per(\G(i);B_{R+1}).
\end{align*}
Observe now that
\begin{multline*}
    \sum_{i\in I_{\E_0}^c}\biggl[ \frac12\Hone\bigl( \{\nu_{\F(i)}=\nu_{\E_0(i)}\}\cap B_{R+1}\setminus\overline{O_{R}} \bigr)
    + \per(\E_0(i); \F(i)^{(1)}\cap B_{R+1}\setminus\overline{O_R}) \biggr] \\
    = \sum_{i\in I_{\E_0}^c}\frac12\per(\E_0(i);B_{R+1}\setminus O_R),
\end{multline*}
which yields
\begin{align*}
    \per(\F'; B_{R+1}) - \per(\F;B_{R+1})  
    & \leq \sum_{i\in I_{\E_0}^c} \biggl( \per(\E_0(i); B_{R+1}\setminus O_R) + \Hone\bigl(\partial O_R \cap \E_0(i)\bigr) - \per(\G(i);B_{R+1})\biggr) \\
    & = \sum_{i\in I_{\E_0}^c} \biggl( \per(\E_0(i)\setminus O_R; B_{R+1}) - \per(\G(i);B_{R+1})\biggr).
\end{align*}
For each $i\in I_{\E_0}^c$ the set $\E_0(i)\setminus O_R$ has boundary of class $C^2$. By a standard result, that we recall in Lemma~\ref{lem:riccardo-fav-lemma} below, any set with boundary of class $C^2$ is a quasi-minimizer of the perimeter for a suitably large constant, depending on the set itself: that is, we can find a constant $C_4>0$, depending ultimately only on $\E_0$ and $\Rba$, such that
\begin{equation} \label{proof:qua-min-5}
\begin{split}
    \per(\F'; B_{R+1}) - \per(\F;B_{R+1})  
    & \leq C_4\sum_{i\in I_{\E_0}^c}\big| \asymm{(\E_0(i)\setminus O_R)}{\G(i)}\big| \\
    & = C_4\sum_{i\in I_{\E_0}^c}\big| |(\E_0(i)\setminus\F(i)) \cap (B_{R+1}\setminus O_R) \big| \\
    & \leq C_4 \sum_{i\in I_{\E_0}^c}\big| \asymm{\E_k(i)}{\F(i)} \big|,
\end{split}
\end{equation}
where we used the fact that $|(\asymm{\E_k(i)}{\E_0(i)})\setminus O_R|=0$ in the last inequality.

By inserting \eqref{proof:qua-min-5} into \eqref{proof:quamin-2} we finally obtain the quasi-minimality inequality \eqref{eq:quasi-min} with $\Lambda\defeq C_3+C_4$ depending only on $\E_0$ and $\Rba$.
\end{proof}

The following property, used in the proof of Theorem~\ref{thm:selection-principle}, is well-known to the experts (see for instance \cite[Lemma~4.1]{AceFusMor13}). Here we state and prove it in the setting of the relative perimeter for the convenience of the reader.

\begin{lemma} \label{lem:riccardo-fav-lemma}
Let $E\subset\R^d$ be an open set with boundary of class $C^2$ and let $\rho>0$. Then there exists a constant $C>0$, depending only on $E$ and $\rho$, such that for every set of finite perimeter $F\subset\R^d$ with $\asymm{E}{F}\ssubset B_\rho$ one has
\begin{equation*}
    \per(E;B_\rho) \leq \per(F;B_\rho) + C |\asymm{E}{F}|.
\end{equation*}
\end{lemma}

\begin{proof}
Let $X\in C^1(\R^d;\R^d)$ be a vector field such that $X=\nu_E$ on $\partial E$ and $\|X\|_\infty\leq1$. Then
\begin{align*}
    \per(E;B_\rho) & - \per(F;B_\rho)
    \leq \int_{\partial E\cap B_\rho} X\cdot\nu_E\dd\Hd - \int_{\partial^*F\cap B_\rho} X\cdot\nu_F\dd\Hd \\
    & = \int_{E\cap B_\rho}\dive X \dd x - \int_{F\cap B_\rho}\dive X \dd x
    \leq \|\dive X \|_{L^\infty(B_\rho)}|\asymm{E}{F}|,
\end{align*}
where we used the fact that, in applying the divergence theorem, all the terms on $\partial B_\rho$ cancel out due to the assumption $\asymm{E}{F}\ssubset B_{\rho}$.
\end{proof}


\subsection{Improved convergence} \label{subsec:improved-conv}
The next step of the strategy exploits the \emph{improved convergence theorem} for clusters by Cicalese, Leonardi and Maggi \cite{CicLeoMag16}, which allows to conclude that the partitions $\E_k$ of the recovery sequence for $\kappa_R(\E_0)$ constructed in Theorem~\ref{thm:selection-principle} are actually smooth perturbations of $\E_0$. Combined with Theorem~\ref{thm:small-asymm}, this reduces the proof of the stability inequality for $\E_0$ to a suitable class of smooth perturbations of $\E_0$. We premise some notation for smooth planar partitions, following \cite{CicLeoMag16}.

\begin{definition} \label{def:smooth-partitions}
A partition $\E=(\E(1),\ldots,\E(N))$ in $\R^2$ is a \emph{$C^{k,\alpha}$-partition} ($k\in\N$, $\alpha\in[0,1]$) if there exist a finite set of points $\{p_i\}_{i\in I}$ and a finite family $\{\gamma_j\}_{j\in J}$ of closed, connected $C^{k,\alpha}$-curves with boundary such that
$$
\partial\E = \bigcup_{j\in J}\gamma_j,
\qquad
\partial^*\E = \bigcup_{j\in J}\mathrm{int}(\gamma_j),
\qquad
\Sigma(\E) = \bigcup_{j\in J}\mathrm{bd(\gamma_j)}=\bigcup_{i\in I}\{p_i\}.
$$
\end{definition}

For a $C^{k,\alpha}$-partition $\E=(\E(1),\ldots,\E(N))$, with $\{p_i\}_{i\in I}$ and $\{\gamma_j\}_{j\in J}$ as in Definition~\ref{def:smooth-partitions}, we say that $f\in C^{k,\alpha}(\pt\E;\R^2)$ if $f\colon \pt\E \to \R^2$ is continuous, $f\in C^{k,\alpha}(\gamma_j;\R^2)$ for every $j\in J$, and
    \[
        \| f \|_{C^{k,\alpha}(\pt\E)} \defeq \sup_{j\in J} \| f \|_{C^{k,\alpha}(\gamma_j)} < \infty.
    \]
Moreover, given two $C^{k,\alpha}$-partitions $\E$ and $\F$, we say that a map $f\colon \pt\E \to \pt\F$ is a \emph{$C^{k,\alpha}$-diffeomorphism between $\partial\E$ and $\partial\F$} if $f$ is a homeomorphism such that $f \in C^{k,\alpha}(\pt \E;\R^2)$, $f^{-1}\in C^{k,\alpha}(\pt\F;\R^2)$, and $f(\Sigma(\E))=\Sigma(\F)$.

Given a partition $\E$ in $\R^2$ and a map $f:\R^2\to\R^2$, we also define the \emph{tangential component of $f$ with respect to $\E$} as the map $\tau_{\E}f:\partial^*\E\to\R^2$ given by
\begin{equation*}
    \tau_{\E}f (x) \defeq f(x) - \bigl(f(x)\cdot\nu_{\E}(x)\bigr)\nu_{\E}(x), \qquad x\in\partial^*\E,
\end{equation*}
where $\nu_{\E}:\partial^*\E\to\mathbb{S}^1$ is any Borel function such that either $\nu_{\E}(x)=\nu_{\E(i)}(x)$ or $\nu_{\E}(x)=\nu_{\E(j)}(x)$ for $x\in\E(i,j)$, $i\neq j$.

With these positions, we can state the main result of this section, which is a direct consequence of the improved convergence theorem in \cite{CicLeoMag16}.

\begin{theorem} \label{thm:improved-conv}
Let $\E_0=(\E_0(1),\ldots,\E_0(N))$ be a uniquely-minimizing locally isoperimetric partition in $\R^2$ with at least two infinite regions, and fix $\Rba>0$.

Then there exist constants $R>0$, $\mu_0>0$, $C_0>0$ (depending only on $\E_0$ and $\Rba$), and a sequence of $C^{1,1}$-partitions $(\F_k)_{k\in\N}\subset\comp_{R}(\E_0)$ such that
\begin{equation} \label{eq:improved-conv}
    \kappa_{\Rba}(\E_0) = \lim_{k\to\infty}\frac{\defi{\F_k}{\E_0}}{\Asymm{\F_k}{\E_0}^2}
    \qquad\text{and}\qquad
    \lim_{k\to\infty}\dis{\F_k}{\E_0}=0,
\end{equation}
and for every $\mu\in(0,\mu_0)$ there exists $k(\mu)\in\N$ and a sequence of $C^{1,1}$-diffeomorphisms $(f_k)_{k\geq k(\mu)}$ between $\partial\E_0$ and $\partial\F_k$ with the following properties:
\begin{enumerate} \setlength\itemsep{6pt}
    \item $\supp(f_k-\id) \ssubset B_R$,
    \item $\| f_k \|_{C^{1,1}(\pt \E_0)} \leq C_0$,
    \item $\| f_k - \id \|_{C^1(\pt \E_0)} \to 0$ as $k\to\infty$,
    \item $\tau_{\E_0}(f_k - \id) =0$ on $\pt \E_0 \setminus I_{\mu}(\Sigma(\E_0))$, where $I_{\mu}(\Sigma(\E_0))\defeq\{x \in \R^2 \colon  \dist(x,\Sigma(\E_0)) < \mu\}$,
    \item $\| \tau_{\E_0}(f_k-\id)\|_{C^1(\pt^*\E_0)} \leq \frac{C_0}{\mu}\| f-\id \|_{C^0(\Sigma(\E_0))}$.
\end{enumerate}
\end{theorem}

\begin{proof}
Let $R>0$ and $(\E_k)_{k\in\N}\subset\comp_{R+1}(\E_0)$ be the sequence of $(\Lambda,r_0)$-minimizing partitions given by Theorem~\ref{thm:selection-principle}. We can find isometries $T_k:\R^2\to\R^2$ such that $\F_k\defeq T_k(\E_k)$ satisfy the conditions \eqref{eq:improved-conv}. Notice that, since $\asymm{\E_k(i)}{\E_0(i)}\ssubset B_{R+1}$ for all $i$, up to taking a larger $R$ we can assume that the same condition is satisfied by $\F_k$.

Since $\E_0$ is in particular a $C^{2,1}$-partition by Theorem~\ref{thm:regularity}, and $(\F_k)_k$ are $(\Lambda,r_0)$-minimizing partitions such that $\dis{\F_k}{\E_0}\to0$, the conclusion follows from \cite[Theorem~1.5]{CicLeoMag16} by noticing that all of the arguments in there are local and can be adapted to our case.
\end{proof}


\section{Stability of the lens among smooth perturbations} \label{sec:step1}

In this section we prove Theorem~\ref{thm:stability-lens} among smooth perturbations of the standard lens partition. In view of Theorem~\ref{thm:improved-conv}, we introduce the following class of smooth perturbations of a locally isoperimetric partition $\E_0$.

\begin{definition} \label{def:epsmuC-perturbation}
    Let $\E_0$ be a locally isoperimetric partition in $\R^2$. Given constants $\e_0>0$ and $R>0$, a $C^1$-partition $\E$ is said to be an \emph{$\e_0$-perturbation of $\E_0$ in $B_R$} if $|\E(i)|=|\E_0(i)|$ for all $i$, and there exists a $C^{1}$-diffeomorphism $\Psi$ between $\pt\E_0$ and $\pt \E$ such that
    \begin{equation*}
        \supp(\Psi-\id) \ssubset B_R
        \qquad\text{and}\qquad
        \| \Psi - \id \|_{C^1(\pt \E_0)} \leq \e_0.
    \end{equation*}
\end{definition}

Taking $\E_0=\lensc$ to be lens partition, its stability among these smooth perturbations is given by the following result.

\begin{theorem}[Stability among $\e_0$-perturbations] \label{thm:stability-smooth}
Let $m>0$ and $R>0$ be fixed. There exist $\e_0>0$ and $\kappa_0>0$ (depending on $m$ and $R$) such that if $\E$ is an $\e_0$-perturbation of $\lensc$ in $B_R$, then
    \[
        \defi{\E}{\lensc} \geq \kappa_0\Asymm{\E}{\lensc}^2.
    \]
\end{theorem}

In order to prove this theorem we will use a geometric fact which states that the image of a graph by a diffeomorphism close to the identity is again a graph. We state this fact as a separate lemma since it could be of independent interest.

\begin{lemma}\label{lem:riccardo-lemma}
Let $f\in C^{1,1}([a,b])$, with $-\infty<a<b<+\infty$. Then, there exist $\lambda_0>0$ and $C>0$, depending only on the $C^{1,1}$-norm of $f$ and on $|b-a|$, with the following property: if $\Psi$ is a $C^1$-diffeomorphism between $\mathrm{graph}(f)$ and its image with $\|\Psi - \mathrm{Id}\|_{C^1(\mathrm{graph}(f))}<\lambda_0$, then there exist $-\infty<c<d<+\infty$ and $\widetilde{f}\in C^1([c,d])$ such that
\begin{equation} \label{proof:graph-1}
\Psi(\mathrm{graph}(f)) = \mathrm{graph}(\widetilde{f})
\qquad\text{and}\qquad \| \widetilde{f} - f\circ \eta \|_{C^1([c,d])} \leq C \|\Psi-\mathrm{Id}\|_{C^1(\mathrm{graph}(f))},
\end{equation}
where $\eta:[c,d]\to [a,b]$ is defined as
\[
\eta(x)\defeq a + \frac{b-a}{d-c}(x-c).
\]
\end{lemma}

\begin{proof}
Using Whitney's Extension Theorem (see \cite[Theorem~2.3]{CicLeoMag16}), we can extend $\Psi$ to a diffeomorphism defined in the entire $\R^2$ with 
\begin{equation} \label{eq:whitney}
\|\Psi-\mathrm{Id}\|_{C^1(\R^2;\R^2)} \leq C \|\Psi-\mathrm{Id}\|_{C^1(\mathrm{graph}(f))} \leq C\lambda_0
\end{equation}
for a constant $C>0$ depending only on $|b-a|$ and on the $C^0$-norm of $f$.
Consider the function $G:[a,b]\to\R$ defined as
\[
G(x)\defeq \Pi_1\left( \Psi(x,f(x)) \right),
\]
where $\Pi_1:\R^2\to\R$ is the projection on the first coordinate. Note that $G\in C^1([a,b])$. By writing $\Psi(x,y) = (\Psi_1(x,y), \Psi_2(x,y))$, so that $G(x)=\Psi_1(x,f(x))$, we easily obtain the following estimates:
\begin{align*}
|G(x) - x|
\leq |(\Psi-\mathrm{Id})(x,f(x))| 
\leq \|\Psi-\mathrm{Id}\|_{C^0(\R^2;\R^2)},
\end{align*}
and
\begin{equation*}
\begin{split}
|G'(x)-1|
& = |\partial_x\Psi_1(x,f(x)) + \partial_y\Psi_1(x,f(x)) f'(x) -1 | \\
& \leq |\partial_x\Psi_1(x,f(x))-1| + |\partial_y\Psi_1(x,f(x))| |f'(x)| \\
& \leq \bigl(1+\|f\|_{C^1([a,b])}\bigr) \|\Psi - \mathrm{Id}\|_{C^1(\R^2;\R^2)}.
\end{split}
\end{equation*}
It follows in particular that 
\begin{equation} \label{eq:est_G-G'}
    \| G - \mathrm{Id} \|_{C^1([a,b])} \leq \bigl(1+\|f\|_{C^1([a,b])}\bigr) \|\Psi - \mathrm{Id}\|_{C^1(\R^2;\R^2)},
\end{equation}
so that by choosing
\begin{equation*}
    \lambda_0 \defeq \frac{1}{2C\bigl(1+\|f\|_{C^1([a,b])}\bigr)}
\end{equation*}
and recalling \eqref{eq:whitney} can guarantee that $G$ is invertible on $[a,b]$ with a $C^1$-inverse $G^{-1}:[c,d]\to[a,b]$, where $c\defeq G(a)$ and $d\defeq G(b)$.

Define $\widetilde{f}:[c,d]\to\R$ as
\[
\widetilde{f}(x) \defeq \Psi_2\left( G^{-1}(x), f(G^{-1}(x)) \right),
\]
so that $\widetilde{f}$ is of class $C^1$ and, by definition, $\Psi(\mathrm{graph}(f)) = \mathrm{graph}(\widetilde{f})$, proving the first condition in \eqref{proof:graph-1}. It remains to estimate $\|\widetilde{f} - f\circ\eta \|_{C^1([c,d])}$. In order to do so, we first observe that, by \eqref{eq:est_G-G'}, we also have
\begin{equation}\label{eq:est_G-1}
\|G^{-1} - \mathrm{Id} \|_{C^1([c,d])}
\leq 2 \|G - \mathrm{Id} \|_{C^1([a,b])}
\leq 2 \bigl(1+\|f\|_{C^1([a,b])}\bigr)\|\Psi-\mathrm{Id}\|_{C^1(\R^2;\R^2)}.
\end{equation}
Moreover, for $x\in[c,d]$ we have
\begin{align*}
|\eta(x) - x|
& = \Big| a + \frac{b-a}{G(b)-G(a)}(x-G(a)) -x \Big| \\
& = \Big| a - G(a) + \Bigl(\frac{b-a}{G(b)-G(a)}-1\Bigr)(x-G(a)) \Big| \\
& \leq |a-G(a)|  + \bigg| \frac{b-a}{G(b)-G(a)}-1 \bigg| |G(b)-G(a)| \\
& \leq \|G-\mathrm{Id}\|_{C^0([a,b])}  + \|G^{-1}-\mathrm{Id}\|_{C^1([c,d])} \|G\|_{C^1([a,b])}|b-a|\\
&\leq c_0 \|\Psi-\mathrm{Id}\|_{C^1(\R^2;\R^2)},
\end{align*}
where last step follows from \eqref{eq:est_G-G'} and \eqref{eq:est_G-1}, for a constant $c_0$ depending on $|b-a|$ and on $\|f\|_{C^1([a,b])}$. Similarly,
\begin{align*}
|\eta'(x) - 1|
= \bigg| \frac{b-a}{G(b)-G(a)} - 1 \bigg|
\leq \|G^{-1}-\mathrm{Id}\|_{C^1([c,d])}
\leq c_0 \|\Psi-\mathrm{Id}\|_{C^1(\R^2;\R^2)}.
\end{align*}
Thus,
\begin{equation}\label{eq:est_eta}
\|\eta - \mathrm{Id}\|_{C^1([a,b])} \leq c_0 \|\Psi-\mathrm{Id}\|_{C^1(\R^2;\R^2)}.
\end{equation}
We are now in position to obtain the desired estimate. For $x\in[c,d]$ we write
\begin{align*}
\widetilde{f}(x) - f\circ\eta(x)
    &= \widetilde{f}(x) - f(G^{-1}(x)) + f(G^{-1}(x)) -  f(\eta(x)) \\
    &= \Pi_2\bigl( (\Psi-\mathrm{Id}) \bigl(G^{-1}(x), f(G^{-1}(x)) \bigl) \bigr) + f(G^{-1}(x)) -  f(\eta(x)),
\end{align*}
from which we get that
\begin{equation*}
\| \widetilde{f} - f\circ\eta \|_{C^0([c,d])}
\leq \|\Psi-\mathrm{Id}\|_{C^0(\R^2;\R^2)}
+ \|f'\|_{C^0([a,b])} \bigl( \|G^{-1} - \mathrm{Id}\|_{C^0([c,d])} + \|\eta - \mathrm{Id}\|_{C^0([c,d])} \bigr),
\end{equation*}
and
\begin{align*}
\| \widetilde{f}' - (f\circ\eta)' \|_{C^0([c,d])}
& \leq \|G^{-1}\|_{C^1([c,d])} \bigl(1 + \|f'\|_{C^0([a,b])}\bigr) \|\Psi-\mathrm{Id}\|_{C^1(\R^2;\R^2)} \\
& \qquad + \| (f'\circ G^{-1}) ((G^{-1})'-1)\|_{C^0([c,d])} + \| (f'\circ\eta)(\eta'-1)\|_{C^0([c,d])} \\
& \qquad + \| f'\circ G^{-1} - f'\circ\eta \|_{C^0([c,d])} \\
& \leq \|G^{-1}\|_{C^1([c,d])} \bigl(1 + \|f'\|_{C^0([a,b])}\bigr) \|\Psi-\mathrm{Id}\|_{C^1(\R^2;\R^2)} \\
& \qquad + \|f\|_{C^{1,1}([a,b])} \bigl( \|G^{-1} - \mathrm{Id} \|_{C^1([c,d])} + \|\eta - \mathrm{Id}\|_{C^1([c,d])} \bigr).
\end{align*}
Thus, combining \eqref{eq:est_G-1} and \eqref{eq:est_eta}, and recalling \eqref{eq:whitney}, we get the second estimate in \eqref{proof:graph-1}.    
\end{proof}

\medskip

\begin{proof}[Proof of Theorem~\ref{thm:stability-smooth}]
Note that by scaling, we can assume that
\begin{equation}\label{eq:r1}
    r_m = 1,
\end{equation}
which, by \eqref{eq:lens-radius}, corresponds to taking $m=\frac{2\pi}{3}-\frac{\sqrt{3}}{2}$. The singular set $\Sigma(\lensc)$ is given by the two points $p_1=(-\sqrt{3}/2,0)$ and $p_2=(\sqrt{3}/2,0)$.
We can also assume without loss of generality that $R\geq R_0$ where $\lens\ssubset B_{R_0}$, since any perturbation in a smaller ball is also a perturbation in a larger ball. We write $\pt \E_0=\gamma_0 \cup \gamma_1 \cup \gamma_2$ with
    \[
        \gamma_0 = \Big\{ (x,0) \colon |x| \geq \sqrt{3}/2\Big\}, \qquad \gamma_i = \Big\{(x,u_i(x)) \colon |x|\leq\sqrt{3}/2 \Big\}, \quad i=1,2, 
    \]
where $u_1(x)=\sqrt{1-x^2}-1/2$ and $u_2(x)=-u_1(x)$.

\smallskip

\noindent \emph{Step 1: Graph representation.} Let $\E$ be an $\e_0$-perturbation of $\lensc$ in $B_R$, with $\e_0\in(0,1)$ to be chosen later, and let $(x_i,y_i)=\Psi(p_i)$, where $\Psi$ is a diffeomorphism as in Definition~\ref{def:epsmuC-perturbation}. We apply a small horizontal translation to $\E$ by $\tau=(-\frac{1}{2}(x_1+x_2),0)$ (notice that $|\tau|\leq\e_0$) and define $\F=\E+\tau$ so that
    \begin{gather*}
         |\F(1)|=|\E(1)|=m, \quad \asymm{\F(i)}{\lensc(i)}\ssubset B_{R+1},\\
         \pt\F = \bigcup_{i=0}^2 \big(\Psi(\gamma_i)+\tau\big), \\ \Sigma(\F)=\{q_1,q_2\} \text{ with }  q_i=\Psi(p_i)+\tau.
    \end{gather*}
Note that with this translation $q_1\cdot e_1 = -q_2\cdot e_1$ (see Figure~\ref{fig:lens-perturbation}). We can then write
    \[
        q_1 = \left(-\frac{\sqrt{3}}{2}(1+\sigma),y_1\right) \qquad \text{ and }\qquad q_2 = \left(\frac{\sqrt{3}}{2}(1+\sigma),y_2\right)
    \]
for some $\sigma\in[-2\e_0,2\e_0]$. Notice that $\F$ is a $(2\e_0)$-perturbation of $\lensc$ in $B_{R+1}$.

We write
\[
\partial\F = \tilde{\gamma}_0 \cup \tilde{\gamma}_1 \cup \tilde{\gamma}_2,
\]
where the former is the boundary between the two unbounded regions, while the other two are the boundaries between the finite region and the upper and the lower unbounded regions, respectively. By Lemma~\ref{lem:riccardo-lemma} we can find $\e_0$ sufficiently small and $C_0>0$, depending on $R$, such that for all $\F$ as before there exist
    \begin{gather*}
        g_0 \in C^1_c\Big((-R-1,R+1)\setminus\big(-\tfrac{\sqrt{3}}{2}(1+\sigma),\tfrac{\sqrt{3}}{2}(1+\sigma)\big)\Big), \nonumber \\
        g_1,g_2 \in C^1\Big(\big[-\tfrac{\sqrt{3}}{2}(1+\sigma),\tfrac{\sqrt{3}}{2}(1+\sigma)\big]\Big), 
    \end{gather*}
with $g_1\geq g_2$,
        \begin{gather}
        g_0\big(\pm(1+\sigma)\tfrac{\sqrt{3}}{2}\big)=g_1\big(\pm(1+\sigma)\tfrac{\sqrt{3}}{2}\big)=g_2\big(\pm(1+\sigma)\tfrac{\sqrt{3}}{2}\big), \label{eq:g-bound-cond}\\
        \int_{-(1+\sigma)\frac{\sqrt{3}}{2}}^{(1+\sigma)\frac{\sqrt{3}}{2}} (g_1-g_2) \dd x = m = \int_{-\frac{\sqrt{3}}{2}}^{\frac{\sqrt{3}}{2}} (u_1-u_2)\dd x \label{eq:g-vol-const},
    \end{gather}
such that
        \begin{equation*}
        \tilde{\gamma}_0 = \Big\{ (x,g_0(x)) \colon |x|\geq(1+\sigma)\tfrac{\sqrt{3}}{2} \Big\},
        \qquad
        \tilde{\gamma}_i = \Big\{ (x,g_i(x)) \colon |x|\leq(1+\sigma)\tfrac{\sqrt{3}}{2} \Big\} \quad\text{for } i=1,2,
        \end{equation*}
and
\begin{equation} \label{eq:g-norm}
\|g_0\|_{C^1\big(\big[-\tfrac{\sqrt{3}}{2}(1+\sigma),\tfrac{\sqrt{3}}{2}(1+\sigma)\big]^c\big)} + \sum_{i=1}^2 \| g_i - \tilde{u}_i \|_{C^1\big(\big[-\tfrac{\sqrt{3}}{2}(1+\sigma),\tfrac{\sqrt{3}}{2}(1+\sigma)\big]\big)} \leq C_0\e_0,
\end{equation}
where we define
    \[
        \tilde{u}_i(x) \defeq (1+\sigma) u_i\left(\frac{x}{1+\sigma}\right), \quad i=1,2.
    \]

\begin{figure}[ht]
\begin{tikzpicture}[scale=2.5,line cap=round,line join=round]
	\draw [thick,name path=A] plot [smooth] coordinates {(-{sqrt(3)/2-0.1},0.1) (-0.55,0.43) (-0.18,0.4) (0.15,0.55) (0.3,0.4) (0.6,0.4) ({sqrt(3)/2+0.1},-0.1)};
    \draw [thick,name path=B] plot [smooth] coordinates {(-{sqrt(3)/2-0.1},0.1) (-0.55,-0.2) (-0.3,-0.5) (0.15,-0.43) (0.4,-0.5) (0.6,-0.25) ({sqrt(3)/2+0.1},-0.1)};
    \tikzfillbetween[of=A and B]{blue!15!white};
	\draw [thick,dashed] (-2,0) -- (-{sqrt(3)/2},0);
	\draw [thick,dashed] ({sqrt(3)/2},0) -- (2,0);
	\draw [thick,dashed] (-{sqrt(3)/2},0) arc (150:30:1);
	\draw [thick,dashed] ({sqrt(3)/2},0) arc (330:210:1); 
	\draw [fill=black] (-{sqrt(3)/2},0) circle (0.75pt);
	\draw [fill=black] ({sqrt(3)/2},0) circle (0.75pt);
    \draw [fill=black] (-{sqrt(3)/2}-0.1,0.1) circle (0.75pt);
    \draw [fill=black] (0.9660254,-0.1) circle (0.75pt);
	\node at (-0.92,-0.1) {\footnotesize{$p_1$}};
	\node at (0.75,0) {\footnotesize{$p_2$}};
    \node at (-0.98,0.2) {\footnotesize{$q_1$}};
	\node at (0.98,-0.2) {\footnotesize{$q_2$}};
    \draw [thick] plot [smooth] coordinates {(-2,0) (-1.96,0) (-1.7,-0.08)  (-1.2,0.05)  (-{sqrt(3)/2-0.1},0.1)};
     \draw [thick] plot [smooth] coordinates {({sqrt(3)/2+0.1},-0.1)  (1.15,0.09)  (1.4,-0.06) (1.6,0.04) (1.95,0) (2,0)};
    \draw (-1.2,0.09) -- (-1.5,0.4);
    \draw (1.16,0.13) -- (1.5,0.4);
    \draw (-0.6,0.45) -- (-0.85,0.5);
    \draw (-0.65,-0.17) -- (-0.85,-0.4);
    \node at (-1.63,0.4) {\footnotesize{$\tilde{\gamma}_0$}};
    \node at (1.63,0.4) {\footnotesize{$\tilde{\gamma}_0$}};
    \node at (-1,0.5) {\footnotesize{$\tilde{\gamma}_1$}};
    \node at (-1,-0.4) {\footnotesize{$\tilde{\gamma}_2$}};
    \node at (0,0) {$\F(1)$};
    \node at (0.9,0.5) {$\F(2)$};
    \node at (0.9,-0.5) {$\F(3)$};
     
\end{tikzpicture}
\caption{A translated smooth perturbation of the standard lens partition $\lensc$.}\label{fig:lens-perturbation}
\end{figure}

\smallskip

\noindent\emph{Step 2: Perimeter deficit.} We now estimate the perimeter deficit between $\E$ and $\E_0$, by exploiting the graph representation obtained in the previous step. For simplicity of notation we call $\ell = \frac{\sqrt{3}}{2}$ and denote by $I_\sigma$ the interval $[-(1+\sigma)\ell,(1+\sigma)\ell]$ so that $I_0=[-\ell,\ell]$. Then we have
    \begin{align*}
    \defi{\E}{\lensc}
        & = \defi{\F}{\lensc} = \per(\F;B_{R+1})-\per(\lensc;B_{R+1}) \\
        & = \int_{I_\sigma^c} \left(\sqrt{1+(g_0^\pr)^2}-1\right) \dd x + \sum_{i=1}^2 \int_{I_\sigma} \sqrt{1+(g_i^\pr)^2} \dd x - \sum_{i=1}^2 \int_{I_0} \sqrt{1+(u_i^\pr)^2} \dd x - 2\sigma\ell \\
        & = \int_{I_\sigma^c} \left(\sqrt{1+(g_0^\pr)^2}-1\right) \dd x
            + \sum_{i=1}^2 \int_{I_\sigma} \left(\sqrt{1+(g_i^\pr)^2}-\sqrt{1+(\tilde{u}_i^\pr)^2}\right) \dd x \\
        & \qquad\qquad\qquad\qquad + \sum_{i=1}^2 \int_{I_\sigma} \sqrt{1+\left(u_i^\pr\left(\frac{y}{1+\sigma}\right)\right)^{2}}\dd y - \sum_{i=1}^2 \int_{I_0} \sqrt{1+(u_i^\pr)^2} \dd x - 2\sigma\ell \\
        & = \int_{I_\sigma^c} \left(\sqrt{1+(g_0^\pr)^2}-1\right) \dd x + \sum_{i=1}^2 \int_{I_\sigma} \left(\sqrt{1+(g_i^\pr)^2}-\sqrt{1+(\tilde{u}_i^\pr)^2}\right)\dd x \\
        & \qquad\qquad\qquad\qquad + \sigma \sum_{i=1}^2 \int_{I_0} \sqrt{1+(u_i^\pr)^2}\dd x - 2\sigma \ell \\
        &\geq \frac{1}{2^{5/2}} \int_{I_\sigma^c} (g_0^\pr)^2 \dd x + \sum_{i=1}^2 \int_{I_\sigma} \frac{\tilde{u}_i^\pr}{\sqrt{1+(\tilde{u}_i^\pr)^2}}(g_i^\pr - \tilde{u}_i^\pr) \dd x \\
        &\qquad\qquad\qquad + c_1 \sum_{i=1}^2 \int_{I_\sigma} (g_i^\pr - \tilde{u}_i^\pr)^2 \dd x + \sigma \sum_{i=1}^2 \int_{I_0} \sqrt{1+(u_i^\pr)^2}\dd x - 2\sigma \ell,
    \end{align*}
where $c_1>0$ is a numeric constant and, in the last inequality, we used the basic estimates 
    \begin{gather*}
        \sqrt{1+t^2}-1 \geq \frac{1}{2^{5/2}} t^2 \quad \text{ for all } |t|\leq 1, \\
        \sqrt{1+t^2}-\sqrt{1+s^2}\geq \frac{s}{\sqrt{1+s^2}}(t-s)+\frac{1}{2\cdot5^{3/2}}(t-s)^2 \quad \text{ for all } |s|,|t|\leq 2,
    \end{gather*}
combined with the fact that $\|g_0^\pr\|_{C^0(I_\sigma^c)} \leq 1$, $\|\tilde{u}_i^\pr\|_{C^0(I_\sigma)}\leq\sqrt{3}$ and $\| g_i'\|_{C^0(I_\sigma)}\leq 2$ for $i=1,2$ by \eqref{eq:g-norm} (up to taking a smaller $\e_0$, if needed).

Observe now that from the explicit form of the lens partition and recalling that $m=\frac{2\pi}{3}-\frac{\sqrt{3}}{2}$ in view of the normalization assumption \eqref{eq:r1}, we easily find
    \[
        \sigma \sum_{i=1}^2 \int_{I_0} \sqrt{1+(u_i^\pr)^2}\dd x - 2\sigma \ell = 2m\sigma.
    \]
Therefore, using this identity, the explicit form of $\tilde{u}_i^\pr$, integrating by parts, and using the boundary conditions \eqref{eq:g-bound-cond} for $g_1$ and $g_2$, we obtain
    \begin{align*}
        \defi{\E}{\lensc} 
        &\geq \frac{1}{2^{5/2}} \int_{I_\sigma^c} (g_0^\pr)^2\dd x + c_1 \sum_{i=1}^2 \int_{I_\sigma} (g_i^\pr - \tilde{u}_i^\pr)^2 \dd x \\
        &\qquad\qquad\qquad -\int_{I_\sigma} \frac{x}{1+\sigma} (g_1^\pr - \tilde{u}_1^\pr) \dd x + \int_{I_\sigma} \frac{x}{1+\sigma} (g_2^\pr - \tilde{u}_2^\pr) \dd x + 2m\sigma \\
        &\geq \frac{1}{2^{5/2}} \int_{I_\sigma^c} (g_0^\pr)^2\dd x + c_1 \sum_{i=1}^2 \int_{I_\sigma} (g_i^\pr - \tilde{u}_i^\pr)^2 \dd x \\
        &\qquad\qquad\qquad +\frac{1}{1+\sigma} \left[ \int_{I_\sigma} (g_1-g_2) \dd x - (1+\sigma)^2 \int_{I_0}(u_1-u_2) \dd x \right] + 2m\sigma \\
        & \xupref{eq:g-vol-const}{=} \frac{1}{2^{5/2}} \int_{I_\sigma^c} (g_0^\pr)^2\dd x + c_1 \sum_{i=1}^2 \int_{I_\sigma} (g_i^\pr - \tilde{u}_i^\pr)^2 \dd x - m\sigma\frac{2+\sigma}{1+\sigma} + 2m\sigma \\
        & = \frac{1}{2^{5/2}} \int_{I_\sigma^c} (g_0^\pr)^2\dd x + c_1 \sum_{i=1}^2 \int_{I_\sigma} (g_i^\pr - \tilde{u}_i^\pr)^2 \dd x + \frac{m\sigma^2}{(1+\sigma)}.
    \end{align*}
Now we extend $\tilde{u}_i$ to zero outside of the interval $I_\sigma$, and define 
    \[
        \tilde{g}_i(x) = 
        \begin{cases}
        g_i(x) &\text{if }x\in I_\sigma,\\
        g_0(x) &\text{if }x\in I_\sigma^c
        \end{cases}
    \]
for $i=1,2$. Then, by Poincar\'{e} inequality, we have for $c_2\defeq\min\{\frac{1}{2^{7/2}},c_1\}$
    \begin{align}
        \defi{\E}{\lensc} &\geq c_2\sum_{i=1}^2 \int_{-R-1}^{R+1} \big[ (\tilde{g}_i-\tilde{u}_i)^\pr\big]^2 \dd x + \frac{m\sigma^2}{1+\sigma} \geq C_R \sum_{i=1}^2 \int_{-R-1}^{R+1} \big(\tilde{g}_i-\tilde{u}_i \big)^2 \dd x + \frac{m\sigma^2}{1+\sigma}  \label{eq:per-deficit}
    \end{align}
for some constant $C_R>0$ depending on $R$.

\smallskip

\noindent \emph{Step 3: Asymmetry estimate.} Let
    \[ 
        \tilde{L}_m \defeq \Bigl\{(x,y)\in \R^2 \colon |x|\leq (1+\sigma)\ell, \ \ |y| \leq (1+\sigma)u_1\Bigl(\frac{x}{1+\sigma}\Bigr) \Bigr\}
    \]
so that $|\tilde{L}_m| = (1+\sigma)^2 |\lens|$. Define the rescaled lens partition $\tilde{\mathcal{L}}_m = (\tilde{L}_m, H^+\setminus\tilde{L}_m,H^-\setminus\tilde{L}_m)$. Then
    \begin{align*}
        \Asymm{\E}{\lensc} &= \Asymm{\F}{\lensc} \leq \frac{1}{2} \sum_{i=1}^3 |\asymm{\F(i)}{\tilde{\mathcal{L}}_m(i)}| + \frac{1}{2} \sum_{i=1}^3 |\asymm{\tilde{\mathcal{L}}_m(i)}{\lensc(i)}| \\
        &\leq \int_{I_\sigma^c} |g_0|\dd x + \sum_{i=1}^2 \int_{I_\sigma} |g_i - \tilde{u}_i|\dd x + \big| |\tilde{L}_m|-|\lens| \big| \\
        &\leq \sum_{i=1}^2 \int_{-R-1}^{R+1} |\tilde{g}_i - \tilde{u}_i| \dd x + |\sigma|(\sigma+2)m.
    \end{align*}
Therefore
    \begin{equation} \label{eq:asym}
        \Asymm{\E}{\lensc}^2 \leq C_R' \left( \sum_{i=1}^2 \int_{-R-1}^{R+1} \big(\tilde{g}_i - \tilde{u}_i\big)^2 \dd x + |\sigma|^2 m^2 \right)
    \end{equation}
for a constant $C_R'>0$ depending only on $R$.

\smallskip

\noindent \emph{Step 4: Conclusion.} Since $|\sigma|\leq2\e_0$, combining the estimates \eqref{eq:per-deficit} and \eqref{eq:asym} we obtain the result of the theorem.
\end{proof}

\begin{proof}[Proof of Theorem~\ref{thm:stability-lens}]
In view of Theorem~\ref{thm:small-asymm}, it is enough to show that the constant $\kappa_R(\lensc)$ defined in \eqref{eq:kappa} is strictly positive, for all $m$ and $R$. Consider the sequence of $C^{1,1}$-partitions $(\F_k)_k$ constructed in Theorem~\ref{thm:improved-conv}, satisfying in particular
\begin{equation} \label{proof:stability}
\kappa_R(\lensc) = \lim_{k\to\infty}\frac{\defi{\F_k}{\E_0}}{\Asymm{\F_k}{\E_0}}.
\end{equation}
For every $\e>0$, each partition $\F_k$ is an $\e$-perturbation of $\lensc$ in $B_{R_0}$ for all $k$ sufficiently large (depending on $\e$), according to Definition~\ref{def:epsmuC-perturbation}, for a suitable radius $R_0>0$ depending only on $m$ and $R$.

In particular we can apply Theorem~\ref{thm:stability-smooth} to $\F_k$ to deduce that the right-hand side of \eqref{proof:stability} is uniformly bounded from below by a positive constant $\kappa_0$, as desired.
\end{proof}


\section{An application: small-mass minimizers for an isoperimetric problem with nonlocal perturbation} \label{sect:nonlocal-lens}

Let $\alpha\in(0,2)$ and $\gamma\geq 0$ be fixed parameters. We consider for $m>0$ the following area-constrained nonlocal isoperimetric problem in $\R^2$:
\begin{equation} \label{nl:min}
\min\Bigl\{ \Fun_\gamma(E) \,:\, E\subset\R^2,\, |E| = m \Bigr\},
\end{equation}
where for every set of finite perimeter $E\subset\R^2$ we define the functional
\begin{equation} \label{nl:fun}
\Fun_\gamma(E)\defeq \per(E) - \Hone\bigl((\partial^*E\cup E^{(1)})\cap H\bigr) + \gamma \int_E\int_E \frac{1}{|x-y|^\alpha}\dd x \dd y.
\end{equation}
Here $H\defeq\{(x,y)\in\R^2:y=0\}$, and we recall that $E^{(\theta)}$ denotes the set of points of Lebesgue density $\theta$ of $E$.
For simplicity of notation we set $E_*\defeq \partial^*E\cup E^{(1)}$ for any set of finite perimeter $E$.

Although this problem can be formulated in any dimension, we restrict here to the case of dimension $d=2$ since our main goal is to show an application of the stability theorem for the lens partition. Nonlocal isoperimetric problems where the perimeter functional is perturbed by a nonlocal repulsive interaction have received great attention in the last decade, particularly in connection with Gamow's liquid drop model for the atomic nucleus (see \cite{ChoMurTop17} for a review), and with the Ohta-Kawasaki model for diblock copolymers. In particular the functional \eqref{nl:fun} is expected to emerge in the scaling limit of a three-phases system for triblock copolymers, in a regime where two majority phases equally occupy nearly all the space forming lamellar structures, and the third minority phase organizes in small droplets on the lamellar flat interfaces, see for instance \cite{AlaBroLuWan21,AlaBroLuWan}.

The connection with the partitioning problem studied in this paper is made clear in the following remark. As a consequence, we obtain that the standard lens $\lens$ (see \eqref{eq:lens}) minimizes the local functional $\Fun_0$, see Remark~\ref{rmk:nonlocal-lens}.

\begin{remark} \label{rmk:nonlocal-1}
For $\gamma=0$, the local functional $\Fun_0$ computed on any bounded set $F\ssubset B_R$, $R>0$, coincides up to a constant (depending on $R$) with the perimeter of the partition $\F\defeq(F,H^+\setminus F, H^-\setminus F)$ associated with $F$ ($H^+$ and $H^-$ denoting the upper and lower half-planes, respectively):
\begin{equation} \label{nl:fun-2}
\Fun_0(F) = \per(\F;B_R) - \Hone(B_R\cap H) \qquad\text{for all }F\ssubset B_R.
\end{equation}
Indeed,
\begin{align*}
\Fun_0(F) + \Hone(B_R\cap H)
& = \per(F) + \Hone(F^{(0)}\cap H\cap B_R) \\
& = \per(F) + \Hone(\partial^*(H^+\setminus F)\cap\partial^*(H^-\setminus F)\cap H\cap B_R)
= \per(\F;B_R).
\end{align*}
We also notice that in view of the identity \eqref{nl:fun-2}, the local functional $\Fun_0$ is lower semicontinuous along any sequence of sets converging in $L^1$ and all contained in a ball of fixed radius, see \cite[Lemma~2.4]{AlaBroVri}.
\end{remark}

\begin{remark} \label{rmk:nonlocal-lens}
For all $m>0$, the standard lens $\lens$ with area $m$ (see \eqref{eq:lens}) minimizes the local functional $\Fun_0$:
\begin{equation} \label{nl:lens-min}
    \mu_0 \defeq \min\bigl\{ \Fun_0(E) \,:\, |E|=1 \bigr\} = \Fun_0(L_1)
\end{equation}
and, by scaling,
\begin{equation} \label{nl:lens-min2}
    \min\bigl\{ \Fun_0(E) \,:\, |E|=m \bigr\} = \Fun_0(\lens) = \mu_0 \sqrt{m}.
\end{equation}
Indeed, for every bounded set $F$ with $|F|=m$, we immediately obtain that $\Fun_0(F)\geq\Fun_0(\lens)$ in view of the relation \eqref{nl:fun-2} and of the minimality of the lens partition $\lensc$. Since we can take a minimizing sequence $(F_k)_k$ for the minimum problem \eqref{nl:lens-min2} made of bounded sets, we have that $\Fun_0(\lens)\leq\Fun_0(F_k)$ for every $k$, so that the minimality of $\lens$ follows.
\end{remark}

In the following theorem we prove existence of minimizers of \eqref{nl:min} for small values of $m$, and that (rescaled) minimizers converge in measure to the standard lens as $m\to0$, as a consequence of the stability property of the lens partition proved in Theorem~\ref{thm:stability-lens}. Notice that by a standard scaling argument, using the homogeneity of the nonlocal kernel, the ``small mass regime'' $m\to0^+$ corresponds to small values of the coefficient $\gamma$ multiplying the nonlocal term; this observation, combined with Remark~\ref{rmk:nonlocal-lens}, provides an heuristic explanation of the following result.

\begin{theorem} \label{thm:nonlocal-lens}
There exists a threshold $m_0>0$, depending only on $\alpha$ and $\gamma$, such that for every $m\in(0,m_0)$ the minimum problem \eqref{nl:min} admits a solution $E_m$. Moreover
\begin{equation} \label{eq:nonlocal-3}
\big| \asymm{ \bigl({\textstyle\frac{1}{\sqrt{m}}}E_m\bigr) }{L_1}\big| \to 0 \qquad\text{as }m\to 0^+,
\end{equation}
where $L_1$ denotes the unit-area lens defined in \eqref{eq:lens}.
\end{theorem}

\begin{proof}
This type of results is by now quite standard for nonlocal isoperimetric problems. We adapt in particular the 2-dimensional existence argument by Kn\"upfer and Muratov for Gamow's liquid drop model, see \cite[Theorem~2.2]{KnuMur13}. We assume in the following $m\leq1$, and we denote along the proof by $C$ a generic, positive constant, depending only on $\alpha$ and $\gamma$, which might change from line to line.

\smallskip
\noindent\textit{Step 1: existence for small $m$.}
Let $(E_k)_k$ be a minimizing sequence for the minimum problem \eqref{nl:min}, and assume without loss of generality that each $E_k$ is the union of finitely many disjoint, open and smooth connected components $E_k = \bigcup_{i=1}^{N_k}E_{k,i}$, $N_k\in\N$, ordered so that $|E_{k,1}|\geq|E_{k,2}|\geq\ldots\geq|E_{k,N_k}|>0$. We can assume that, for $k$ large,
\begin{equation} \label{nl:proof1}
\Fun_\gamma(E_k) \leq \Fun_{\gamma}(\lens)
= \mu_0\sqrt{m} + \gamma m^{\frac{4-\alpha}{2}}\int_{L_1}\int_{L_1}\frac{1}{|x-y|^\alpha}\dd x \dd y
= \mu_0\sqrt{m} \Bigl( 1 + C m^{\frac{3-\alpha}{2}}\Bigr)
\end{equation}
(or else $\lens$ would already be a minimizer), where we used \eqref{nl:lens-min2}.

Suppose now $N_k>1$, so that $|E_{k,i}|\leq\frac{m}{2}$ for all $i=2,\ldots,N_k$. By minimality of $\lens$ for the local functional $\Fun_0$ and positivity of the nonlocal term we have for all $i\in\{1,\ldots,N_K\}$
\begin{equation} \label{nl:proof2}
\Fun_\gamma(E_k) \geq \Fun_0(E_{k,i}) + \Fun_0(E_k\setminus E_{k,i})
\xupref{nl:lens-min2}{\geq} \mu_0\bigl(|E_{k,i}|^{1/2} + (m-|E_{k,i}|)^{1/2}\bigr).
\end{equation}
By combining \eqref{nl:proof1} and \eqref{nl:proof2}, squaring both sides, we find for $i=2,\ldots,N_k$
\begin{align*}
    2|E_{k,i}|^{1/2}(m-|E_{k,i}|)^{1/2} \leq Cm^{4-\alpha} + C m ^{\frac{5-\alpha}{2}} \leq C m^{\frac{5-\alpha}{2}},
\end{align*}
from which it follows, using also $|E_{k,i}|\leq\frac{m}{2}$,
\begin{equation} \label{nl:proof3}
    |E_{k,i}| \leq C m^{4-\alpha} \qquad\text{for all }i=2,\ldots, N_k.
\end{equation}
Define the set
\begin{equation*}
  F_k\defeq \lambda (E_k\setminus E_{k,N_k}), \qquad\text{with }\lambda\defeq\biggl(\frac{m}{m-|E_{k,N_k}|}\biggr)^\frac12 \in (1,\sqrt{2}],
\end{equation*}
so that $|F_k|=m$. We have
\begin{align*}
    \Fun_{\gamma}(F_k) - \Fun_{\gamma}(E_k)
    & = \lambda\Bigl[\per(E_k)-\Hone((E_k)_*\cap H)\Bigr] - \lambda\Bigl[\per(E_{k,N_k})-\Hone((E_{k,N_k})_*\cap H)\Bigr] \\
    & \qquad  + \gamma\lambda^{4-\alpha}\int_{F_k}\int_{F_k}\frac{1}{|x-y|^\alpha}\dd x \dd y - \Fun_{\gamma}(E_k) \\
    & \leq (\lambda^{4-\alpha}-1)\Fun_{\gamma}(E_k) - \lambda \Fun_0(E_{k,N_k}) \\
    & \leq (\lambda^{4}-1)\Fun_{\gamma}(E_k) - \Fun_0(E_{k,N_k}) \\
    & \leq \frac{6}{m}|E_{k,N_k}|\Fun_{\gamma}(E_k) - \mu_0 |E_{k,N_k}|^{1/2} \\
    & \leq \mu_0 |E_{k,N_k}|^{1/2}\Bigl(C\Fun_\gamma(E_k)m^{\frac{2-\alpha}{2}} - 1 \Bigr),
\end{align*}
where we used in particular \eqref{nl:lens-min2} in the third inequality, and \eqref{nl:proof3} in the last one. Since $\Fun_{\gamma}(E_k)$ is uniformly bounded by a constant depending only on $\alpha$ and $\gamma$ (recall that $m\leq1$), we can find $m_0\in(0,1)$, also depending only on $\alpha$ and $\gamma$, such that the previous quantity is negative for all $m\in(0,m_0)$ and for all $k$. Therefore $\Fun_{\gamma}(F_k) < \Fun_{\gamma}(E_k)$, that is, by removing the last connected component of $E_k$ and rescaling we reduce the energy. By iterating this argument $N_k-1$ times, removing a connected component at each step, we replace $E_k$ by a connected set $G_k$ such that $|G_k|=m$ and $\Fun_{\gamma}(G_k)\leq \Fun_{\gamma}(E_k)$.

In particular, $(G_k)_k$ is a minimizing sequence for \eqref{nl:min} made of connected sets. Moreover, the sets $G_k$ have equibounded perimeter: this follows from the estimate
\begin{equation} \label{nl:proof4}
    \Fun_0(E) \geq \frac12\per(E)
\end{equation}
for every set of finite perimeter $E$, which can be proved as follows. By \cite[Proposition~19.22]{Maggi} one has $\per(F;H^\pm)>\per(F;H)$ for all $F\subset H^\pm$ with finite perimeter and finite measure. Hence, using also \cite[Theorem~16.3]{Maggi},
\begin{align*}
    \per(E)
    & = \per(E;H^+) + \per(E;H^-) + \Hone(\partial^*E\cap H) \\
    & \geq \per(E\cap H^+;H) + \per(E\cap H^-;H) + \Hone(\partial^*E\cap H) \\
    & = 2\Hone(E^{(1)}\cap H) + 2\Hone(\partial^*E\cap H)
    = 2 \Hone(E_*\cap H),
\end{align*}
from which the estimate \eqref{nl:proof4} follows.

Since we are in dimension $d=2$ and $(G_k)_k$ are connected sets with uniformly bounded perimeter, we have that $\sup_k\mathrm{diam} (G_k)<\infty$ \cite[Remark 12.28]{Maggi}. By applying horizontal translations (notice that the functional $\Fun_\gamma$ is invariant with respect to horizontal translations) we can then assume that the sets $G_k$ are contained in a fixed ball of large radius. A standard compactness argument, combined with the lower semicontinuity of $\Fun_{\gamma}$ with respect to $L^1$-convergence (see Remark~\ref{rmk:nonlocal-1}), yields the existence of a minimizer in \eqref{nl:min} for all $m\in(0,m_0)$.

\smallskip
\noindent\textit{Step 2: convergence to the lens.}
Let $E_m$ be a minimizer in \eqref{nl:min} for $m\in(0,m_0)$, and define $\widetilde{E}_m\defeq \frac{1}{\sqrt{m}}E_m$, so that $|\widetilde{E}_m|=1$.

Notice that every minimizer $E_m$ is necessarily connected, in the sense that it cannot be written as disjoint union $E_m=A\cup B$ of two sets of positive Lebesgue measure in such a way that $\per(E_m)=\per(A)+\per(B)$. Indeed, if not then one could horizontally translate one of the two components far apart from the other, without changing the local energy $\Fun_0$ but strictly decreasing the nonlocal energy.

The sets $\widetilde{E}_m$ have equibounded perimeter, since by scaling and comparing with $\lens$
\begin{align*}
    \per(\widetilde{E}_m)
    & = \frac{1}{\sqrt{m}}\per(E_m)
    \xupref{nl:proof4}{\leq}\frac{2}{\sqrt{m}}\Fun_0(E_m)
    \leq \frac{2}{\sqrt{m}}\Fun_\gamma(E_m)
    \leq \frac{2}{\sqrt{m}}\Fun_\gamma(\lens)
    \leq 2\Fun_\gamma(L_1).
\end{align*}
Therefore, by connectedness, there exists $R_0>0$ such that $\sup_{m <m_0}\mathrm{diam}(\widetilde{E}_m)\leq R_0$. Associate with $\widetilde{E}_m$ the partition $\E_m\defeq(\widetilde{E}_m,H^+\setminus\widetilde{E}_m,H^-\setminus\widetilde{E}_m)$, and notice that $\E_m\in\comp_{R_0}(\mathcal{L}_1)$. By applying an horizontal translation, we can assume that
\begin{equation*}
    \Asymm{\E_m}{\mathcal{L}_1} = \dis{\E_m}{\mathcal{L}_1} = |\asymm{\widetilde{E}_m}{L_1}|.
\end{equation*}
Then by the stability of $\mathcal{L}_1$ proved in Theorem~\ref{thm:stability-lens} we have
\begin{align*}
    \kappa_{1,R_0}|\asymm{\widetilde{E}_m}{L_1}|^2
    & \leq \per(\E_m;B_{R_0}) - \per(\mathcal{L}_1;B_{R_0})
    \xupref{nl:fun-2}{=} \Fun_0(\widetilde{E}_m) - \Fun_0(L_1) \\
    & = \frac{1}{\sqrt{m}} \Bigl( \Fun_0(E_m) - \Fun_0(\lens) \Bigr) \\
    & \leq \frac{\gamma}{\sqrt{m}} \biggl( \int_{\lens}\int_{\lens}\frac{1}{|x-y|^\alpha}\dd x \dd y - \int_{E_m}\int_{E_m}\frac{1}{|x-y|^\alpha}\dd x \dd y \biggr) \\
    & \leq \frac{C\gamma}{\sqrt{m}}|\asymm{E_m}{\lens}|
    = C\gamma\sqrt{m}|\asymm{\widetilde{E}_m}{L_1}|,
\end{align*}
where the second inequality follows by minimality of $E_m$, and the third inequality is a standard Lipschitz continuity estimate of the nonlocal energy, see for instance \cite[Equation~(3.2)]{KnuMur13}. Hence \eqref{eq:nonlocal-3} follows.
\end{proof}

\begin{remark} \label{rmk:gen-min}
For values of $m$ above a suitable threshold, it is expected that minimizers of \eqref{nl:min} fail to exist, since the nonlocal part of the energy becomes dominant and can be decreased by splitting a set into two parts and moving them far apart from each other (by horizontal translations). However, it is possible to prove existence of \emph{generalized minimizers} for all $m$, following for instance the approach in \cite{NovPra21} for a general Gamow's model. By \emph{generalized minimizer} of \eqref{nl:min} we mean a collection of sets of finite perimeter $(E_1,\ldots,E_M)$, $M\in\N$, such that
\begin{equation*}
    \sum_{i=1}^M|E_i|=m \qquad\text{and}\qquad \inf\Bigl\{ \Fun_\gamma(E) \,:\, E\subset\R^2,\, |E| = m \Bigr\}=\sum_{i=1}^M\Fun_\gamma(E_i).
\end{equation*}
We also remark that one can prove analogous results for more general kernels in the nonlocal energy (as those considered in \cite{NovPra21}). We will not further investigate these problems here.
\end{remark}


\appendix
\section{A ``volume-fixing variation'' lemma} \label{sect:appendix}

We state here a variant of a key result in the theory of minimizing clusters (see for instance \cite[Section~29.5]{Maggi}), which allows to exchange volumes between the different chambers of a given partition through local deformations, with a control on the corresponding perimeter variation. For our purposes (see in particular the proof of Theorem~\ref{thm:selection-principle}), we need to make sure that the perturbation is compactly supported in a fixed open set.

\begin{lemma}[Volume-fixing variation] \label{lem:volume-fixing}
Let $\E_0=(\E_0(1),\ldots,\E_0(N))$ be a locally isoperimetric partition in $\R^2$, and let $O_{R}$ be the set constructed in Definition~\ref{def:eye}. Let also $T:\R^2\to\R^2$ be an isometry such that $T(\E_0(i))\subset O_{R}$ for all $i\in I_{\E_0}$.

Then there exist constants $C_0>0$, $\e_0>0$, and $r_0>0$ (depending on $T(\E_0)$ and $R$) with the following property.

If $\E$ and $\G$ are $N$-partitions such that $\dis{\E}{T(\E_0)}<\e_0$ and $\asymm{\G(i)}{\E(i)}\ssubset B_{r_0}(x)\ssubset B_{R+1}$ for some $x\in\R^2$, then there exists a $N$-partition $\widetilde{\G}$ such that:
\begin{enumerate}\setlength\itemsep{5pt}
    \item $\asymm{\widetilde{\G}(i)}{\G(i)} \ssubset O_{R} \setminus\overline{B_{r_0}(x)}$ for all $i\in\{1,\ldots,N\}$,
    \item $|\widetilde{\G}(i)|=|\E(i)|$ for all $i\in\{1,\ldots,N\}$,
    \item $|\per(\widetilde{\G};B_{R+1})-\per(\G;B_{R+1})| \leq C_0\per(\E;B_{R+1})\sum_{i\in I_{\E_0}}\big| |\G(i)|-|\E(i)| \big|$,
    \item $|\dis{\widetilde{\G}}{\E}-\dis{\G}{\E}| \leq C_0\per(\E;B_{R+1})\sum_{i\in I_{\E_0}}\big| |\G(i)|-|\E(i)| \big|$.
\end{enumerate}
\end{lemma}

\begin{proof}
This result is proved in \cite[Corollary~29.17]{Maggi} for clusters in $\R^d$. The presence of multiple regions with infinite measure in our case does not affect the proof, as their volumes do not need to be preserved and, furthermore, everything is localized in a large ball $B_{R+1}$. We need only to enforce the condition $\asymm{\widetilde{\G}(i)}{\G(i)} \ssubset O_{R}$.

The proof is based on \cite[Theorem~29.14]{Maggi}, where one selects two finite families $\{y_\alpha\}_{\alpha=1}^M$ and $\{z_\alpha\}_{\alpha=1}^M$ of interface points of $T(\E_0)$, and constructs the required perturbation by modifying $\G$ either in $\bigcup_{\alpha=1}^M B_{\e_1}(y_\alpha)$ or in $\bigcup_{\alpha=1}^M B_{\e_1}(z_\alpha)$, for some $\e_1>0$, by means of suitable diffeomorphisms.

Therefore, to ensure that the part outside $O_R$ is unchanged, we only need to make sure that it is possible to choose the points $y_\alpha$ and $z_\alpha$ inside $O_R$. This is guaranteed by \cite[Remark~29.15]{Maggi}, which only requires that $|T(\E_0(i))\cap O_R|>0$ for all $i\in\{1,\ldots,N\}$.

With these considerations, the proof can be adapted to deal with our situation. The only condition not addressed in \cite{Maggi} is the estimate (iv). Its proof is discussed in \cite[Appendix~B]{CicLeoMag16}.
\end{proof}

\bigskip
\subsection*{Acknowledgments}
MB is member of the GNAMPA group of INdAM. MB is partially supported by the project PRIN 2022PJ9EFL ``Geometric Measure Theory: Structure of Singular Measures, Regularity Theory and Applications in the Calculus of Variations'', funded by the European Union - Next Generation EU, Mission 4 Component 2 - CUP:E53D23005860006. RC was partially supported under NWO-OCENW.M.21.336, MATHEMAMI - Mathematical Analysis of phase Transitions in HEterogeneous MAterials with Materials Inclusions. IT was partially supported by the Simons Foundation (851065), by the NSF-DMS (2306962), and by the Alexander von Humboldt Foundation.


\begin{thebibliography}{HMRR02}

\bibitem[ABLW]{AlaBroLuWan}
Stanley Alama, Lia Bronsard, Xinyang Lu, and Chong Wang, \emph{Decorated phases
  in a model of tri-block co-polymers}, in preparation.

\bibitem[ABLW21]{AlaBroLuWan21}
\bysame, \emph{Periodic minimizers of a ternary non-local isoperimetric
  problem}, Indiana Univ. Math. J. \textbf{70} (2021), no.~6, 2557--2601.

\bibitem[ABV22]{AlaBroVri2}
Stanley Alama, Lia Bronsard, and Silas Vriend, \emph{On a free-endpoint
  isoperimetric problem in $\mathbb{R}^2$}, INdAM Workshop: Anisotropic
  Isoperimetric Problems \& Related Topics, Springer, 2022, pp.~151--167.

\bibitem[ABV23]{AlaBroVri}
\bysame, \emph{The standard lens cluster in $\mathbb{R}^2$ uniquely minimizes
  relative perimeter}, Preprint (2023), arXiv:2307.12200.

\bibitem[AFM13]{AceFusMor13}
Emilio Acerbi, Nicola Fusco, and Massimiliano Morini, \emph{Minimality via
  second variation for a nonlocal isoperimetric problem}, Comm. Math. Phys.
  \textbf{322} (2013), no.~2, 515--557.

\bibitem[BN24]{BroNov}
Lia Bronsard and Michael Novack, \emph{An infinite double bubble theorem},
  Preprint (2024), arXiv:2401.08063.

\bibitem[CL12]{CicLeo12}
Marco Cicalese and Gian~Paolo Leonardi, \emph{A selection principle for the
  sharp quantitative isoperimetric inequality}, Arch. Ration. Mech. Anal.
  \textbf{206} (2012), no.~2, 617--643.

\bibitem[CLM16]{CicLeoMag16}
Marco Cicalese, Gian~Paolo Leonardi, and Francesco Maggi, \emph{Improved
  convergence theorems for bubble clusters {I}. {T}he planar case}, Indiana
  Univ. Math. J. \textbf{65} (2016), no.~6, 1979--2050.

\bibitem[CLM17]{CicLeoMag17}
\bysame, \emph{Sharp stability inequalities for planar double bubbles},
  Interfaces Free Bound. \textbf{19} (2017), no.~3, 305--350.

\bibitem[CM16]{CarMag16}
Marco Caroccia and Francesco Maggi, \emph{A sharp quantitative version of
  {H}ales' isoperimetric honeycomb theorem}, J. Math. Pures Appl. (9)
  \textbf{106} (2016), no.~5, 935--956.

\bibitem[CMT17]{ChoMurTop17}
Rustum Choksi, Cyrill~B. Muratov, and Ihsan Topaloglu, \emph{An old problem
  resurfaces nonlocally: {G}amow's liquid drops inspire today's research and
  applications}, Notices Amer. Math. Soc. \textbf{64} (2017), no.~11,
  1275--1283.

\bibitem[FAB{\etalchar{+}}93]{FoiAlfBroHodZim93}
Joel Foisy, Manuel Alfaro, Jeffrey Brock, Nickelous Hodges, and Jason Zimba,
  \emph{The standard double soap bubble in {${\bf R}^2$} uniquely minimizes
  perimeter}, Pacific J. Math. \textbf{159} (1993), no.~1, 47--59.

\bibitem[FM12]{FusMor12}
Nicola Fusco and Massimiliano Morini, \emph{Equilibrium configurations of
  epitaxially strained elastic films: second order minimality conditions and
  qualitative properties of solutions}, Arch. Ration. Mech. Anal. \textbf{203}
  (2012), no.~1, 247--327.

\bibitem[FMP08]{FusMagPra08}
Nicola Fusco, Francesco Maggi, and Aldo Pratelli, \emph{The sharp quantitative
  isoperimetric inequality}, Ann. of Math. (2) \textbf{168} (2008), no.~3,
  941--980.

\bibitem[Fus15]{Fus15}
Nicola Fusco, \emph{The quantitative isoperimetric inequality and related
  topics}, Bull. Math. Sci. \textbf{5} (2015), no.~3, 517--607.

\bibitem[HMRR02]{HutMorRitRos02}
Michael Hutchings, Frank Morgan, Manuel Ritor\'{e}, and Antonio Ros,
  \emph{Proof of the double bubble conjecture}, Ann. of Math. (2) \textbf{155}
  (2002), no.~2, 459--489.

\bibitem[KM13]{KnuMur13}
Hans Kn\"{u}pfer and Cyrill~B. Muratov, \emph{On an isoperimetric problem with
  a competing nonlocal term {I}: {T}he planar case}, Comm. Pure Appl. Math.
  \textbf{66} (2013), no.~7, 1129--1162.

\bibitem[LM17]{LeoMag17}
Gian~Paolo Leonardi and Francesco Maggi, \emph{Improved convergence theorems
  for bubble clusters. {II}. {T}he three-dimensional case}, Indiana Univ. Math.
  J. \textbf{66} (2017), no.~2, 559--608.

\bibitem[Mag08]{Mag08}
Francesco Maggi, \emph{Some methods for studying stability in isoperimetric
  type problems}, Bull. Amer. Math. Soc. (N.S.) \textbf{45} (2008), no.~3,
  367--408.

\bibitem[Mag12]{Maggi}
\bysame, \emph{Sets of finite perimeter and geometric variational problems},
  Cambridge Studies in Advanced Mathematics, vol. 135, Cambridge University
  Press, Cambridge, 2012.

\bibitem[MN22]{MilNee}
Emanuel Milman and Joe Neeman, \emph{The structure of isoperimetric bubbles on
  $\mathbb{R}^n$ and $\mathbb{S}^n$}, Preprint (2022), arXiv:2205.09102.

\bibitem[NP21]{NovPra21}
Matteo Novaga and Aldo Pratelli, \emph{Minimisers of a general {R}iesz-type
  problem}, Nonlinear Anal. \textbf{209} (2021), Paper No. 112346, 27.

\bibitem[NPT23]{NovPaoTor}
Matteo Novaga, Emanuele Paolini, and Vincenzo~Maria Tortorelli, \emph{Locally
  isoperimetric partitions}, Preprint (2023), arXiv:2312.13709.

\bibitem[Pas25]{Pas}
Giulio Pascale, \emph{Existence and nonexistence of minimizers for classical
  capillarity problems in presence of nonlocal repulsion and gravity},
  Nonlinear Anal. \textbf{251} (2025), Paper No. 113685.

\bibitem[PP24]{PasPoz}
Giulio Pascale and Marco Pozzetta, \emph{Quantitative isoperimetric
  inequalities for classical capillarity problems}, Calc. Var. Partial
  Differential Equations \textbf{63} (2024), no.~9, Paper No. 225, 49.

\bibitem[Rei08]{Rei08}
Ben~W. Reichardt, \emph{Proof of the double bubble conjecture in {$\bold
  R^n$}}, J. Geom. Anal. \textbf{18} (2008), no.~1, 172--191.

\bibitem[RHLS03]{ReiHeiLaiSpi03}
Ben~W. Reichardt, Cory Heilmann, Yuan~Y. Lai, and Anita Spielman, \emph{Proof
  of the double bubble conjecture in {${\bf R}^4$} and certain higher
  dimensional cases}, Pacific J. Math. \textbf{208} (2003), no.~2, 347--366.

\bibitem[Wic04]{Wic04}
Wacharin Wichiramala, \emph{Proof of the planar triple bubble conjecture}, J.
  Reine Angew. Math. \textbf{567} (2004), 1--49.

\end{thebibliography}
\end{document}